\documentclass[12pt,reqno]{amsart}
 
\usepackage{latexsym}
\usepackage{amssymb, tikz}
\usepackage{graphicx}
\usepackage{url}

\renewcommand{\epsilon}{\varepsilon}

\newcommand{\R}{{\mathbb R}}

\newcommand{\Z}{{\mathbb Z}}

\renewcommand{\phi}{\varphi}

\newcommand{\gcal}{\mathcal{G}}

\numberwithin{equation}{subsection}
\newtheorem{theo}[equation]{{\sc Theorem}}

\newtheorem{lem}[equation]{{\sc Lemma}}

\newtheorem{prop}[equation]{{\sc Proposition}}

\theoremstyle{definition}
\newtheorem{defn}[equation]{{\sc Definition}}

\theoremstyle{remark}
\newtheorem{rem}[equation]{Remark}

\title{Smooth billiards with a large Weyl remainder}
\author{Suresh Eswarathasan, Iosif Polterovich and John A.~Toth}
\address{Department of Mathematics and
Statistics, McGill University, 805 Sherbrooke Str. West,
Montr\'eal QC H3A 2K6, Ca\-na\-da.}
\address{Centre de Recherches Mathematiques, Univer\-sit\'e de Mont\-r\'eal, CP 6128 succ.
Centre-Ville, Mont\-r\'eal QC  H3C 3J7, Canada.}
\email{suresh@math.mcgill.ca \\ eswarath@crm.umontreal.ca}

\address{D\'e\-par\-te\-ment de math\'ematiques et de
sta\-tistique, Univer\-sit\'e de Mont\-r\'eal, CP 6128 succ.
Centre-Ville, Mont\-r\'eal QC  H3C 3J7, Canada.}
\email{iossif@dms.umontreal.ca}

\address{Department of Mathematics and
Statistics, McGill University, 805 Sherbrooke Str. West,
Montr\'eal QC H3A 2K6, Ca\-na\-da.}
\email{jtoth@math.mcgill.ca} 

\date{}

\begin{document}
\begin{abstract}
The celebrated Hardy--Landau lower bound for the error term in the Gauss's circle problem can be viewed as an estimate from below 
for the remainder in Weyl's law on a square,  with either Dirichlet or Neumann boundary conditions.  
We prove an analogous estimate for smooth star-shaped planar domains admitting an appropriate one-parameter family  of periodic billiard trajectories.  
Examples include ellipses and smooth domains of constant width. 
Our results confirm a prediction of P.~Sarnak who proved a similar statement  for surfaces without boundary.  We also obtain lower bounds on the error term  in higher dimensions. In this case, the main contribution to the Weyl remainder typically comes from the ``big'' singularity at zero of the wave trace.  However, for certain domains, such as the Euclidean ball, the dimension of the family of periodic trajectories is large enough to  dominate the contribution of the singularity at zero.

\end{abstract}
\maketitle
\section{Introduction and main results}
\subsection{Weyl's law}
Let $\Omega \subset \R^n$ be a smooth bounded Euclidean domain. 
Denote by  $\{\phi_{\lambda_j}\}_{j \in \mathbb{N}}$ the  $L^2$-normalized eigenfunctions of the Laplacian on $\Omega$ with either Dirichlet or Neumann boundary conditions, and let $\lambda_j^2$ be the corresponding eigenvalues.
 
Denote by $ N(\lambda)= \# \{ \lambda_j \leq \lambda \}$  the spectral counting function. 
The {\it Weyl remainder} $R(\lambda)$ is defined by the following equality:
\begin{equation} \label{2term}
N(\lambda) = \frac{\text{vol}(\Omega)}{(4 \pi)^{\frac{n}{2}} \Gamma\left(\frac{n}{2}+1\right)} \lambda^n \pm \frac{\text{vol}(\partial \Omega)}{2^{n+1} \pi^{\frac{n-1}{2}}\Gamma\left(\frac{n+1}{2}\right)} \lambda^{n-1} + R(\lambda),
\end{equation}
where $"+"$  in front of the second term corresponds to Neumann boundary conditions and $"-"$ to Dirichlet boundary conditions.
According to Weyl's conjecture, 
 \begin{equation}\label{2term2}
R(\lambda) = o(\lambda^{n-1}) \,\,\, \text{as} \,\,\, \lambda \rightarrow \infty, \end{equation}
and therefore \eqref{2term} is a two-term asymptotic formula. 
Weyl's conjecture was proved in \cite{Ivr} under the nonperiodicity condition: periodic trajectories of the billiard flow  $G^t:S^*\Omega \rightarrow S^*\Omega$ form a  set of measure zero in the unit cotangent bundle $S^*\Omega$. It was conjectured by Ivrii  that the non-periodicity condition holds for all Euclidean billiards. We refer to \cite{Vas, Vas1, PS, SV} for the proofs of this conjecture in some special cases and further discussion.

The remainder $R(\lambda)$ is known to  depend on dynamical properties of the billiard  flow. In the present paper we focus on lower bounds for
$R(\lambda)$. For manifolds without boundary, related results were obtained in \cite{S, Kar, JP, JPT}.  Let us also mention the celebrated Hardy--Landau lower bound for the error term in the Gauss's circle problem.  It can be  viewed as a lower bound on the Weyl remainder for a flat square torus, as well as  for a square with either Neumann or Dirichlet boundary conditions.
\subsection{Main result}
\label{sec:main}
It is well known that the growth of the error term is closely related to the singularities of the wave trace $\sum_{j=1}^{\infty} e^{-i\lambda_j t}$ when understood as a distribution. The wave trace has a ``big'' singularity at $t=0$, 
and other singularities at $\pm t \in \mathcal{L}(\Omega)$, where $\mathcal{L}(\Omega)$ denotes the length spectrum of $\Omega$, i.e. the set of lengths of closed billiard trajectories.  In any dimension, the singularity at zero produces  the two main terms in Weyl's law.  In two dimensions, the influence of the singularity at zero on the error term $R(\lambda)$ in the two-term asymptotics is insignificant, and therefore in order to estimate the remainder from below one has to study  the contribution of the 
periodic orbits.  This makes the two-dimensional case the most interesting and difficult.  
\subsubsection*{Notation}
Throughout the paper we use two different representations of the billiard flow.  On one hand, we view it  as a map on $S^*\Omega$, the unit cotangent bundle of $\Omega$. On the other hand, we consider the {\it billiard ball map} $\beta$  (see, for instance, \cite{GM2})  on $B^*\partial \Omega$, the co-ball bundle of the boundary.  We set $B^* \partial \Omega = \{ (q,\xi) \in T^* \partial \Omega;   |\xi| < 1 \}$ and let $\pi: B^* \partial \Omega \rightarrow \partial \Omega$ be the canonical projection map $(q,\xi) \mapsto q.$  We note that $B^*\partial \Omega$ is the open co-ball bundle excluding the glancing set $S^*\partial \Omega.$ The co-ball bundle can be identified with $S_{\partial \Omega}^{+}(\Omega)$ (the superscript $+$ stands for directions pointing inside $\Omega$) using the map 
\begin{equation}
\label{zetaproj}
\zeta: B^* \partial \Omega \rightarrow S_{\partial \Omega}^+ \Omega,
\end{equation}
 given by  $\zeta(q,\xi) = (q, \xi + \sqrt{1-|\xi|^2} \nu_q^+)$, where $\nu_q^+$ denotes the internal unit normal to $\partial \Omega$ at the point $q$.   Further on, maps and functions on the lifted space $S_{\partial \Omega}^*\Omega$  and the unit tangent bundle $S^*\Omega$ will be distinguished from the corresponding objects on $B^*\partial \Omega$ by the addition of  a tilde~$\tilde \, .$ In particular, if $\tilde{U} \subset S_{\partial \Omega}^+\Omega$,  the set $U \subset B^*\partial \Omega$ is defined by $\zeta^{-1}(\tilde{U})=U$.
 
An element $(q,\eta) \in S_{\partial \Omega}^*\Omega$ uniquely determines a billiard trajectory emanating from the point $q \in \partial \Omega$ in the direction $\eta$.  Given a set $\tilde{\Lambda}_{\partial\Omega} \subset S_{\partial \Omega}^*\Omega$ we denote by  
$\tilde{\Lambda} \subset S^*\Omega$ a family of billiard trajectories with initial data in $\tilde{\Lambda}_{\partial\Omega}$.  

\smallskip
Recall the definition of a clean (in the sense of Bott) fixed point set for a diffeomorphism $\Phi$ (see \cite{DG}).
\begin{defn}
\label{clean}
Let $M$ be a manifold and $\Phi: M \rightarrow M$ be a diffeomorphism.  A submanifold $Z \subset M$ of fixed points of $\Phi$ is called {\em clean} if for each $z \in Z$, the kernel of the map $(d \Phi_z - I): T_z M \rightarrow T_z M$ equals $T_z Z$.
\end{defn}

In order to formulate our main result we need the following definition.

\begin{defn} 
\label{microclean} 
Assume that the billiard ball map $\beta$ on a smooth bounded planar domain $\Omega$ possesses a smooth invariant circle $\Lambda_T  \subset B^*\partial \Omega$  consisting of the fixed points of $\beta^k$ for some given $ k\ge 2$.  Let $\tilde{\Lambda}_T \subset S^*\Omega$ be the corresponding  family of periodic billiard trajectories of length $T$. We say that this family  is {\it admissible} if the following properties are  satisfied:

\smallskip

\noindent (i) {\it (isolation in the length spectrum)}   There exists $\epsilon >0$ with the property that there are no periodic billiard trajectories
of period $t\in[T-\epsilon, T+\epsilon]$ other than those in $\tilde{\Lambda}_T \cup \tilde{\Lambda}_T^{-1}$. Here  $\tilde{\Lambda}_T^{-1}$  denotes 
the family of  the same trajectories as in $\tilde{\Lambda}_T$ but traversed in the opposite direction; in other words,    $\tilde{\Lambda}_T$ is mapped into $\tilde{\Lambda}_T^{-1}$ by the  isomorphism of $S^*\Omega$ given by $(x,\eta) \to (x,-\eta)$.

\smallskip

\noindent (ii)  {\it (cleanliness)} The set  $\Lambda_T \subset B^*\partial\Omega$ is a clean submanifold of the fixed points of~$\beta^k$. 

\smallskip

\noindent (iii) {\it (separation from the glancing set)}  The set  $\Lambda_T$ lies in the interior of $B^*\partial \Omega$.
\end{defn}
Condition (i) is the most difficult one to check for particular domains. It is used  to separate the contribution of $\Lambda_T$ to the growth of the remainder $R(\lambda)$  from the contributions of all the other periodic orbits.  Note that, a priori, contributions of different periodic  orbits may cancel each other, particularly due to the presence of the Maslov factors (cf.  Remark \ref{remark:Maslov}). Condition (i) could be probably made less restrictive, but in the present form it  is sufficient  to make sure that the contribution of $\Lambda_T$ is  not cancelled out.
 
Condition (ii) is standard (see \cite{DG});  it  is necessary in order  to apply the stationary phase argument.  

Condition (iii)  means  that the invariant circle $\Lambda_T$ stays away from the glancing set.  This is needed to construct a microlocal parametrix of the Dirichlet wave operator (see section~\ref{multilink}). It also ensures  that the billiard ball map  $\beta$ is smooth on $\Lambda_T$. Note that for nonconvex domains, the billiard ball map  $\beta$ has discontinuities due to the existence of trajectories touching the boundary from inside (see \cite{Tab}).

The result below was  predicted by  P. Sarnak in 1995 and  can be viewed as  an analogue of \cite[Proposition 3.1]{S} for planar domains. Given $\alpha >0$, we  use the standard notation $f(\lambda) \gg~\lambda^{\alpha}$, if there exist positive constants $\lambda_0$ and $C$ such that $f(\lambda) \ge C\,\lambda^{\alpha}$ for any
$\lambda>\lambda_0$.
\begin{theo} \label{mainthm}
Let $\Omega$ be a bounded, star--shaped  planar domain with  smooth boundary,  having  an admissible 
family of periodic billiard trajectories in the sense of Definition \ref{microclean}. Then the Weyl remainder for either Dirichlet or Neumann boundary conditions satisfies
\begin{equation}
\label{main:bound}
\frac{1}{\lambda}\int_{\lambda}^{2 \lambda} |R(\tau)| \ d\tau \gg \sqrt{\lambda}.
\end{equation}
\end{theo}
The existence of a family of periodic  billiard trajectories of the same length is not a generic assumption (see \cite{PS}).  At the same time, it is  satisfied for a number of  interesting examples, including   ellipses and smooth domains of constant width. There are many other billiards with one-parameter families of closed trajectories (see \cite{BZ} and references therein) for which Theorem \ref{mainthm} would hold provided the admissibility conditions (particularly, isolation in the length spectrum) are satisfied.

\begin{rem}
 On any smooth strictly convex billiard,    there exist  caustics, corresponding to the invariant circles of the billiard ball map (see \cite{Laz} for some related spectral results). However, these invariant circles  typically have irrational rotation numbers and do not give rise to a family of periodic  orbits. It is not clear whether they are  relevant  for  the study of the Weyl remainder. 
\end{rem}

\begin{rem}
\label{Nostadium}
Note that partially rectangular domains (including the Bunimovich stadium, for which \eqref{main:bound} was originally conjectured by Sarnak)  are not covered by Theorem \ref{mainthm}. The difficulties arising in their study are twofold. First, note that the fixed point set corresponding to the bouncing ball trajectories is not a submanifold of $B^*\partial\Omega$; it has a boundary corresponding to the endpoints of the flats, and therefore the standard clean intersection calculus does not work. Second, most partially rectangular domains are not smooth. For example, the  boundary of the Bunimovich stadium is  $C^{1,1}$ at the points where the circular arcs meet the flat parts. Despite the fact that these singularities are rather  mild,  the wave equation approach seems hard to implement. This is due to diffraction of singularities associated with incident waves hitting corner points and the fact that the bouncing-ball flats $\Lambda_T$ pass through the corners at the interface between  the rectangular part and the semicircular wings. Consequently, instead of attempting to show that the  conormal expansions for $E^{b}(t)$ at $t=T$ are still valid, the stationary Balian-Bloch approach for the Dirichlet (or Neumann) resolvent $(-\Delta_{\Omega} - (\lambda + i0)^{2} )^{-1}$ could  be used to derive these expansions. 
This program will be implemented  in the subsequent paper \cite{PT}.
\end{rem}

\begin{rem}
\label{noweyl}
It is not assumed in Theorem \ref{mainthm} (and in Theorem \ref{higher} below) that the Weyl remainder satisfies the  upper bound \eqref{2term2} (we note that $R(\lambda) \neq o(\lambda)$ does not automatically imply \eqref{main:bound} because the latter is an estimate on average).
The nonperiodicity condition is known to be true for convex analytic domains \cite{Vas, SV},
and hence \eqref{2term2} holds for ellipses and analytic domains of constant width (existence of such domains has been shown in \cite{Weg}).
\end{rem}

\subsection{Higher dimensions}
\label{higher:sec1}
 Let us now assume that $\Omega$ is a Euclidean domain of dimension three or higher. In this case, a  lower bound on the error term could be proved using the contribution of the singularity at zero of the wave trace. In particular, we have the following
\begin{theo}
\label{higher}
Let $\Omega \subset \mathbb{R}^n$, $n\ge 3$ be a Euclidean domain with smooth boundary. Suppose that the total mean curvature of the boundary 
$\int_{\partial \Omega} H \,  \neq 0$. Then  the Weyl remainder for either Dirichlet or Neumann boundary conditions satisfies
\begin{equation}
\label{higher:bound}
\frac{1}{\lambda}\int_{\lambda}^{2 \lambda} |R(\tau)| \ d\tau \gg\lambda^{n-2}.
\end{equation}
\end{theo}
Theorem \ref{higher} is a corollary of the more general Proposition \ref{Riesz} proved in Section \ref{higher:sec}.
Note that, unlike the assumptions of Theorem \ref{mainthm}, the condition $\int_{\partial \Omega} H \,  \neq 0$ is satisfied generically.
In particular, it holds for any convex domain $\Omega$. 

The lower bound \eqref{higher:bound} could be  generalized to Riemannian manifolds with boundaries, see Proposition \ref{Riesz}. 
It can be also extended to some Euclidean domains with nonsmooth boundaries.
For instance, it holds for generic polyhedra, including all convex ones (see Remark \ref{Fed}). 
Estimates  on the remainder in the lattice-counting problem in  higher-dimensional bodies \cite{Wal, BG} show that   in dimensions 
$n\ge 5$, the bound  \eqref{higher:bound} is sharp and is saturated for rectangular parallelepipeds.

An analogue of Theorem \ref{mainthm} also  holds in dimensions $n\ge 3$. Using \eqref{pretrace} and following the same argument as in the proof of Theorem \ref{mainthm}, one could show:
\begin{equation}
\label{dyn:higher}
\frac{1}{\lambda}\int_{\lambda}^{2 \lambda} |R(\tau)| \ d\tau \gg \ \lambda^{d/2},
\end{equation}
where  $d$ is the dimension of  an  admissible compact submanifold of $B^*\partial \Omega$ corresponding to the periodic billiard trajectories of the same length (note that 
in \eqref{main:bound} one has $d=\dim \Lambda_T=1$).
According to the nonperiodicity conjecture, $d/2<n-1$,  and hence $d\le 2n-3$.  If this inequality is strict, estimate  \eqref{dyn:higher} does not give anything new compared to \eqref{higher:bound}. For instance, the bouncing ball orbits in a rectangular parallelepiped form a fixed point set of dimension $d=n-1$, and therefore
\eqref{dyn:higher}  coincides with \eqref{higher:bound} if $n=3$ and is weaker for $n \ge 4$.
If the equality $d=2n-3$ holds,  estimate \eqref{dyn:higher} gives a better bound than \eqref{higher:bound}. As will be shown in  subsection \ref{balls}, 
this is the case for Euclidean balls, and therefore the following result holds:
\begin{prop}
\label{ball:bound}
The Weyl remainder  $R(\lambda)$  on a $n$-dimensional ball  $\mathbb{B}^n \subset \mathbb{R}^n$ satisfies
\begin{equation*}
\frac{1}{\lambda}\int_{\lambda}^{2 \lambda} |R(\tau)| \ d\tau \gg \ \lambda^{n-\frac{3}{2}}, \quad n=1,2,\dots
\end{equation*}
\end{prop}

\subsection{Outline of the proof of Theorem \ref{mainthm}}  To illustrate the main ideas of the proof of Theorem \ref{mainthm}, we sketch the proof here for the Dirichlet boundary conditions in the special case when $\Omega$ is a smooth strictly star--shaped planar domain with an admissible family of closed billiard trajectories $\Lambda_T \subset B^*\partial \Omega$ satisfying an additional assumption
\begin{equation} \label{simplify}
\Lambda_T \cap   0_{B^*\partial \Omega}  = \emptyset.
\end{equation} 
Here, the zero-section $0_{B^*\partial \Omega}$ corresponds in $S_{\partial \Omega}^*\Omega$ to the inward pointing  unit conormal directions to the boundary.   One can check that an ellipse satisfies these assumptions, see Section \ref{ellipse}. The fact that the orbits belonging to $\Lambda_T$ are not orthogonal to the boundary $\partial\Omega$ significantly simplifies the analysis (see subsection \ref{discussion}).

The proof  has two main ingredients. The first is the {\it  Rellich identity} (\ref{Rellich}) for Dirichlet eigenvalues which gives an explicit formula for the trace $\sum_{j} \rho(\lambda_j - \lambda)$ with $\rho \in {\mathcal S}(\R)$ as a weighted boundary trace of the normal derivatives of the Dirichlet eigenfunctions. Here and further on, $\mathcal{S}(\R)$ denotes the Schwartz space of rapidly decaying functions. 
The second ingredient is a conormal expansion for the boundary  trace of the wave operator at microlocally clean fixed point sets of the billiard flow $G^t$, with period $t=T$, due to Hezari and Zelditch \cite{HZ}.  

Let $\nu_{q}$ be the unit outward-pointing boundary normal at $q \in \partial \Omega.$ 
We define the usual  normalized Dirichlet eigenfunction boundary traces to be 
 \begin{align} \label{bdytraces}
 u_{j}^b(q)&:= \frac{1}{\lambda_j} \partial_{\nu} \phi_{\lambda_j}(q) \text{ with } q \in \partial \Omega.
 \end{align}
 The Rellich identity for Dirichlet eigenvalues  \cite{Re} says that $2 \lambda_j^{2} = \int_{\partial \Omega} \langle \nu_{q}, q \rangle \,\, | \partial_{\nu} \phi_{\lambda_{j}}(q)|^2 \, d\sigma(q)$. When written in terms of the normalized boundary traces in (\ref{bdytraces}), this becomes
\begin{equation} \label{Rellich}
\int_{\partial \Omega} \langle \nu_{q}, q \rangle \,\, | u_{j}^b(q)|^2 \, d\sigma(q) =  2. \end{equation}
The simple identity in (\ref{Rellich}) yields the following suggestive formula for the spectral counting function,
\begin{equation} 
N(\lambda) = \frac{1}{2} \sum_{\lambda_j \leq \lambda} \int_{\partial \Omega} \langle \nu_{q}, q \rangle \,\, | u_{j}^b(q)|^2 \, d\sigma(q). \end{equation}

Without loss of generality we may assume that $\Omega$ is strictly star-shaped with respect to  the origin $0 \in \Omega$. Then the  weight function 
$F(q) = \langle q, \nu_{q} \rangle > 0 $ for any $q \in \partial \Omega$ (cf. \cite{BH}).
Similarily, (\ref{Rellich}) also implies that for  $\rho_T \in S(\R)$ with $\hat{\rho_T} \in C^{\infty}_{0}(\R),$ $\hat{\rho_T}(T) = 1$ and  $0 \notin \text{supp} \,  \hat{\rho_T},$
\begin{equation} \label{key}
\sum_{j=1} \rho_T(\lambda-\lambda_j) =  \frac{1}{2}  \int_{\R} \hat{\rho_T}(t)  e^{i t \lambda} \, Tr_{\partial \Omega} E_F^{b}(t) \, dt. \end{equation}
In $(\ref{key})$, $E_F^b(t):C^{\infty}(\partial \Omega) \to C^{\infty}(\partial \Omega)$ is the {\em weighted} boundary trace operator with Schwartz kernel
\begin{equation} \label{bdywave}
E_F^{b}(t,q,q') = \sum_{j=1}^{\infty} e^{- i \lambda_j t}  \, \sqrt{\langle q, \nu_q \rangle} \sqrt{\langle q',\nu_{q'}}\rangle \,  u_j^b(q) \overline{u_j^b(q')} \text{ with }(q,q') \in \partial \Omega \times \partial \Omega, \end{equation}
and $Tr_{\partial \Omega} E_F^b(t) = \int_{\partial \Omega} E_F^{b}(t,q,q) d\sigma(q).$

It is important to note that while the LHS of (\ref{key}) is the usual interior trace, the RHS involves only a weighted boundary trace, {\em not} the more complicated interior counterpart. Indeed, as is well known, the Dirichlet half wave operator $U_{\Omega}(t) = e^{-it \sqrt{ \Delta_{\Omega}} }$ is quite complicated even for small times. Thus, instead of attempting to adapt  the wave-trace methods in the boundaryless case directly to manifolds with boundary,  we use the Rellich formula (\ref{Rellich}) to reduce the asymptotic analysis of $N(\lambda)$ and the corresponding remainder $R(\lambda)$ entirely to the boundary, $\partial \Omega$. In view of (\ref{key}) and under the assumption (\ref{simplify}), we apply conormal expansions for $Tr_{\partial \Omega}E^b(t)$   to get lower bounds for $R(\lambda).$ Indeed, with $\hat{\rho_T} \in C^{\infty}_{0}(\R)$ as above and  under the assumption (\ref{simplify}), it is proved in \cite[Lemma 7]{HZ} that \begin{equation}  \label{conormal}
Tr_{\partial \Omega}E_F^{b}(t) = \int_{\partial \Omega}E_F^b(t,q,q) \ d \sigma(q) \sim_{t \rightarrow T}  \Big( \int_{B^*\partial \Omega} a_0(q)  \gamma(q,\eta) dq d\eta \Big) (t-T+i0)^{-\frac{3}{2}}.  \end{equation}
Here, 
$$ a_{0}(q) = c_{\Lambda_T}  \langle q, \nu_q \rangle$$ with  $c_{\Lambda_T} \neq 0$ and $\gamma(q,\eta) = \sqrt{1- |\eta|^2_g}.$ 
Strictly speaking, the weight function $F(q) = \langle q, \nu_q \rangle$ is not present in \cite{HZ} and the analysis is for the sine kernel $\frac{ \sin t \sqrt{\Delta_{\Omega}}  }{ \sqrt{\Delta_{\Omega}} },$ but the expansion in (\ref{conormal}) follows in exactly the same way.
Since the  weight function $ F(q) = \langle q, \nu_{q} \rangle > 0 $ for any $q \in \partial \Omega$, and so it follows that  $a_{0}(q) \neq 0$ for any $q \in \partial \Omega.$  Substitution of the expansion (\ref{conormal}) into (\ref{key}) gives
\begin{equation} \label{upshot1} 
\sum_j \rho_T(\lambda_j-\lambda) \sim_{\lambda \rightarrow \infty}  C_0  e^{i T \lambda} \sqrt{\lambda}.\end{equation}
Moreover, since $a_0(q) \neq 0$  for all $q \in \partial \Omega$ in (\ref{conormal}), it follows that $ \int_{B^*\partial \Omega} a_0(q)  \gamma(q,\eta) dq d\eta  \neq 0$ and consequently, $C_0 \neq 0$ in (\ref{upshot1}).

 The last step is, modulo some technical modifications, similar to the analysis in the boundaryless case. Namely, using (\ref{2term}) and the support properties of $\hat{\rho_T},$ we show that
 \begin{equation} \label{upshot2}
 \sum_{j} \rho_T(\lambda_j - \lambda) = \int_{0}^{\infty} \rho_T(\tau-\lambda) dN(\tau)  = \int_{0}^{\infty} \rho_T(\tau -\lambda) dR(\tau). \end{equation}
Thus, for a smooth strictly star-shaped domain $\Omega$ satisfying (\ref{simplify}) it then follows from (\ref{upshot2}) and (\ref{upshot1}) that 
\begin{equation}\label{upshot3}
  \int_{0}^{\infty} \rho_T(\tau -\lambda) dR(\tau) \sim_{\lambda \to \infty} C_0 e^{iT\lambda} \sqrt{\lambda}. \end{equation}
 Integration  by parts  and an application of Sarnak's calculus lemma (see Section \ref{sarnak}) completes the proof. 

\subsection{Discussion} 
\label{discussion}
As we mentioned in the previous subsection, once isolation and cleanliness are established (see  subsection \ref{ellipse}), the above argument can be applied without changes if $\Omega$ is an ellipse. However,  in some other examples,  such as domains of  constant width, the assumption \eqref{simplify} is certainly {\em not} satisfied. In fact, in the latter case
the {\em entire} fixed point set of bouncing-ball geodesics is co-normal to the boundary; that is,  $ ( \tilde{ \Lambda}_T \cap T^*_{\partial \Omega} \Omega )\subset N^*\partial \Omega.$
  Consequently, the proof in \cite[Lemma 7]{HZ} of the conormal expansion in (\ref{conormal}) does not apply directly, since  the boundary restriction operator $\gamma_{\partial \Omega}: C^0(\Omega) \to C^0(\partial \Omega)$ (see \cite[Section 2.2]{HZ}) is not a homogeneous Fourier integral operator (FIO) in the sense of H\"{o}rmander microlocally near $0_{T^*\partial \Omega} \times N^*\partial \Omega$. To compensate, we  prove a semiclassical variant of (\ref{conormal}) using the calculus of compactly-supported semiclassical Fourier integral operators \cite{GuSt}. 
  

We  note that variants of the Rellich identity in (\ref{Rellich}) have appeared  in several  articles related to other eigenfunction phenomena such as ``Quantum Ergodic Restriction"  \cite{CTZ}, $L^2$ restriction bounds for the Dirichlet boundary traces \cite{HT}, completeness of boundary traces \cite{HHHZ},  and control estimates for eigenfunctions \cite{BZ1,BHW}.  For the purposes of this paper, one could alternatively avoid use of the Rellich identity and instead apply the Possion summation formula of Guillemin-Melrose \cite{GM} for the interior wave trace at non-trivial periods $T \neq 0.$ However, we find the approach via the Rellich identity for the boundary traces very convenient and  follow it throughout the paper.  

\subsection{Plan of the paper}
The plan of the paper is as follows. In Section \ref{higher:sec} we present Sarnak's calculus lemma and prove Theorem~\ref{higher}. The rest of the article is devoted to the proof of Theorem \ref{mainthm} and its applications.  In Section \ref{micro}, we review some basic semiclassical analysis that is used later on. Section \ref{multilink} provides  background on billiard dynamics; we introduce there certain canonical relations relevant to our analysis. 
In Section \ref{guts}, we derive a conormal expansion for the weighted boundary wave trace microlocalized to a neighbourhood of a clean fixed point set of the billiard flow. 
The proof of Theorem \ref{mainthm} is presented in Section~\ref{mainproof}. Combining the conormal expansion for  the weighted boundary wave trace  with Sarnak's calculus lemma, we  prove Theorem \ref{mainthm} for the  Dirichlet boundary conditions in Subsection \ref{sec:dirichlet}.   
The changes needed to prove Theorem \ref{mainthm}  in the case of Neumann boundary conditions are presented in Subsection \ref{Neumann}.
Finally,  in Section \ref{examples}, we discuss some examples for which the  assumptions in Theorem \ref{mainthm} are satisfied.

\subsection*{Acknowledgments}  We are grateful to L. Hillairet for several comments and corrections regarding earlier versions of the paper, and to L. Polterovich for numerous helpful discussions. We would also like to thank  K.~Bezdek, V.~Kaloshin, S. Nonnenmacher, L.~Parnovski, V.~Petkov, D. Sher, S. Tabachnikov, D. Vassiliev, J. Wunsch and S. Zelditch for useful remarks, as well as M. Levitin and  L.~Houde Therrien for producing the figures.
This research has been supported in part by NSERC, FRQNT and Canada Research Chairs Program. 
While working on the paper, the first named author was a postdoctoral fellow at the CRM and a visitor at the IHES.
The hospitality of these institutions is gratefully acknowledged. 

\section{Lower bounds in higher dimensions}
\label{higher:sec}
\subsection{Sarnak's calculus lemma} 
\label{sarnak}
The following lemma is used in the proofs of both Theorems \ref{mainthm} and \ref{higher}. It is a straightforward generalization 
of  \cite[Lemma 5.1]{S}.

\begin{lem} \label{calclemma}
Let $R(t)$ be a locally integrable function on $\mathbb{R}$ satisfying $R(t)=\mathcal{O}(|t|^s)$ for some $s>0$ .  Assume that, for some Schwartz function $\psi \in {\mathcal S}(\R)$  and some $\alpha >0$, we have
\begin{equation} \label{calcint}
\left| \int_{-\infty}^{\infty} R(t) \psi(x-t) dt \right| \gg x^\alpha.
\end{equation}
Then 
\begin{equation*}
\frac{1}{X} \int_{X}^{2X} |R(t)| \ dt  \gg X^\alpha.
\end{equation*}
\end{lem}
\begin{proof} We essentially follow the argument in \cite[Section 5.1]{LPS}. 
The hypothesis $R(t)=~\mathcal{O}(|t|^s)$ implies that $\int_{-\infty}^{\infty} |R(t) \psi(x-t)| dt < \infty$.  Integrating both sides of $\int_{-\infty}^{\infty} |R(t) \psi(x~-t)| dt \gg x^\alpha$ with respect to $x$ from $5X/4 $ to $7X/4 $ leads to 
\begin{equation} \label{lowerbnd}
\int_{-\infty}^{\infty} \bigg [ \int_{5X/4}^{7X/4} | \psi(x-t) | \ dx \bigg ] | R(t) | \ dt \gg  X^{\alpha+1}.
\end{equation}
Let us split the integral in $t$ into three parts and rewrite the left hand side as
\begin{multline*}
\int_{-\infty}^{X} \int_{5X/4}^{7X/4} | \psi(x-t) | | R(t) | dx \  dt + \int_{X}^{2X} \int_{5X/4}^{7X/4} | \psi(x-t) | | R(t) | dx \ dt+\\ \int_{2X}^{\infty} \int_{5X/4}^{7X/4} | \psi(x-t) | | R(t) | dx \ dt.
\end{multline*}
Since $\psi$ is rapidly decaying and $R(t)$ grows at most polynomially, the first and the third integrals above are $\mathcal{O}(1)$ as $X \to \infty.$ At the same time, there exists a a constant $C>0$ such that 
$\int_{5X/4}^{7X/4} | \psi(x-t)| dx \le C $ uniformly for $t\in [X, 2X]$.   Taking these estimates into account and using \eqref{lowerbnd} we get:
$$
C\,\int_X^{2 X} |R(t)| dt \ge \int_{X}^{2X}\bigg [ \int_{5X/4}^{7X/4}  | \psi(x-t) | \ dx \bigg ] | R(t) | dt \gg X^{\alpha+1}.
$$
Dividing by $X$ completes the proof of the lemma.
\end{proof}


\begin{rem} A straightforward modification of  the proof of Lemma \ref{calclemma} shows that instead of averaging over the interval of length $X$  one could average over an interval of length  $X^{\beta}$ for any $\beta>0$.
As a consequence, the lower bounds proved in Theorems \ref{mainthm} and \ref{higher}, as well as in Proposition \ref{ball:bound}, hold on average for any interval of length $\lambda^\beta$, $\beta>0$. 

\end{rem}

\subsection{Proof of Theorem \ref{higher}}  Let $M$ be a smooth, compact, $n$-dimensional Riemannian manifold with boundary, and 
let $\psi \in {\mathcal S}(\mathbb{R})$ with $\widehat \psi \in C_0^{\infty}(\mathbb{R})$,
$\widehat \psi =1$ near zero and such that the support of $\psi$ is sufficiently small. Then, as was shown in \cite{Ivr} (see also \cite{Me, Saf}),
\begin{equation}
\label{ivrii} 
\int_{-\infty}^\infty N(\lambda) \psi (\lambda-t) dt = \sum_{j=0}^\infty a_j(M)\lambda^{n-j} + \mathcal{O}(\lambda^{-\infty}),
\end{equation}
where $N(\lambda)$ is the eigenvalue counting function for either Dirichlet or Neumann boundary conditions. The coefficients $a_j(M)$  are geometric invariants of the Riemannian manifold $M$ and its boundary $\partial M$. In general, they depend on the choice of the boundary conditions, 
and differ by nonzero multiplicative constants from the corresponding heat invariants of $M$, see \cite[Section 5.1]{JPT}.
In particular, $a_0(M)$ is proportional to the volume of $M$,  $a_1(M)$ is proportional to the length of the boundary $M$. The coefficients $a_0(M)$ and $a_1(M)$ 
correspond to the main terms in Weyl's law. Formula \eqref{ivrii} implies the following asymptotics for the remainder:
\begin{equation}
\label{WR}
\int_{-\infty}^\infty R(\lambda) \psi (\lambda-t) dt = a_2(M) \lambda^{n-2} +\mathcal{O}(\lambda^{n-3}),
\end{equation}
where
\begin{equation}
\label{a2}
a_2=C \left(\int_M \mathcal{K} + 2n \int_{\partial M} H\right).
\end{equation}
Here $C$ is some nonzero constant, $\mathcal{K}$ is the {\it scalar curvature} of $M$ and $H$ is the {\it mean curvature} of the boundary $\partial M$ \cite{MS, BG}. Note that the formula for $a_2(M)$ is the same for both Dirichlet and Neumann boundary conditions.
We have the following
\begin{prop} 
\label{Riesz}
Let $M$ be a smooth compact $n$-dimensional Riemannian manifold with boundary such that
$$\int_M \mathcal{K} + 2n \int_{\partial M} H \neq 0,$$
where $\mathcal{K}$ and $H$ denote the scalar and the mean curvature, respectively.
Then for either Dirichlet or Neumann boundary conditions,  the Weyl remainder $R(\lambda)$ satisfies
the lower bound
$$
\frac{1}{\lambda}\int_{\lambda}^{2 \lambda} |R(\tau)| \ d\tau \gg\lambda^{n-2}.
$$ 
\end{prop}
\begin{proof} By  formula \eqref{a2}, we have $a_2(M) \neq 0$. The result is then a direct consequence
of Lemma~\ref{calclemma} and  formula \eqref{WR}.
\end{proof}
Theorem \ref{higher} follows immediately from Proposition \ref{Riesz} by observing that for Euclidean domains $\mathcal{K} \equiv0$.
\begin{rem}
\label{Fed}
It follows from the results of \cite{Fedosov} that formula \eqref{ivrii} holds for Euclidean polyhedra. Therefore, lower bound \eqref{higher:bound} is also valid 
in this case, provided the corresponding coefficient $a_2$ is nonzero. This is true generically --- in particular, for any convex polyhedron, as follows from the explicit formula for $a_2$ given in \cite{Fedosov}. 
Note also that the nonperiodicity condition is always satisfied for polyhedra \cite{Vas1}, and hence the remainder estimate \eqref{2term2} holds.
\end{rem}
\begin{rem} Proposition \ref{Riesz} admits the following straightforward generalization. Let $s=\min\{j\ge 2| \,  a_j(M) \neq 0\}$. Then
$$
\frac{1}{\lambda}\int_{\lambda}^{2 \lambda} |R(\tau)| \ d\tau \gg\lambda^{n-s}.
$$ 
\end{rem}

\smallskip
\begin{rem}
Sections 3--6 below  are devoted to the proof of Theorem \ref{mainthm}. While the result is concerned with planar domains, 
we present the arguments  in arbitrary dimension.  In particular, this justifies  formula \eqref{dyn:higher} which is used in subsection \ref{balls}. 
\end{rem}
\section{Semiclassical Pseudodifferential and Fourier Integral Operators} \label{micro}

We briefly review here the relevant calculus of semiclassical pseudodifferential and Fourier integral operators that will be used in the proof of Theorem \ref{mainthm} in Section \ref{guts}. The reader is referred to \cite{Zw}  for further details.
\subsection{Semiclassical pseudodifferential operators ($h$-PsiDOs)}  Let $M$ be a $C^{\infty}$ compact manifold.  The basic semiclassical  symbol spaces are
\begin{eqnarray} \label{symbol}
\nonumber S^{m,k}_{cl}(T^*M): &=& \{ a \in C^{\infty}(T^*M \times (0,h_0]); \\
&& \hskip-3truecm a \sim_{h \rightarrow 0} h^{-m} \sum_{j=0}^{\infty} a_{k-j}(x,\xi) h^{j}  \text{ with } | \partial_{x}^{\alpha} \partial _{\xi}^{\beta} a_{k-j}(x,\xi)| \leq C_{\alpha,\beta} \langle \xi \rangle^{k-|\beta |} \}.  \end{eqnarray}
Here, we use the standard notation for  $\langle \xi \rangle : = \sqrt{ 1 + |\xi|^2}.$ The corresponding space of {\em{$h$-pseudodifferential operators}} is 
\begin{equation} \label{psdo}
\Psi^{m,k}_{cl}:= \{ A_{h}: C^{\infty}(M) \rightarrow C^{\infty}(M); A_{h}= Op_{h}(a) \, \text{with}\, a \in S^{m,k}_{cl}(T^*M) \}, \end{equation}

\noindent where the Schwartz kernels are locally of the form
$$Op_h(a)(x,y) = (2\pi h)^{-n} \int_{\R^n} e^{i \langle x-y,\xi \rangle/h} a(x,\xi,h) \, d\xi, \text{ for } a \in S^{m,k}_{cl}.$$
Given $A_{h} \in \Psi^{m,k}_{cl},$ the principal symbol is defined as $\sigma(A_{h}) := h^{-m} a_0$ using the notation of (\ref{symbol}).
It is sometimes convenient to use the $h$-Weyl quantization with corresponding Schwartz kernel
\begin{equation} \label{weyl}
Op_h^w(a)(x,y) = (2\pi h)^{-n} \int_{\R^n} e^{i \langle x-y,\xi \rangle/h} a \left( \frac{x+y}{2},\xi,h \right) \, d\xi, \text{ for } a \in S^{m,k}_{cl}. \end{equation}

One useful feature of the latter quantization scheme is that for $a(x,\xi,h)$ real-valued, the corresponding Weyl quantization is formally self-adjoint with
$Op_h^w(a)^* = Op_h^w(a).$ The operators $A_h \in \bigcup_{j,k \in \mathbb Z} \Psi^{j,k}_{cl}$ form an algebra; in particular,   given $A_{h} \in \Psi^{m,k}_{cl} \text{ and } B_{h} \in \Psi^{m',k'}_{cl},$ the composition
$ A_h \circ B_h \in \Psi^{m+m',k+k'}_{cl}$ has the principal symbol $ \sigma( A_h \circ B_h)(x;\xi) = \sigma(A_h)(x,\xi) \cdot \sigma(B_h)(x,\xi) = h^{-m-m'} a_0(x,\xi) b_0(x,\xi)$ in terms of the local representation in (\ref{psdo}).

The  sequence  of eigenfunctions $(\phi_h)_{h \in (0,h_0]}$ has compact $h$-wave front supported in an arbitrarily small neighborhood of the characteristic variety  $ S^*M =\{ (x,\xi) \in T^*M; |\xi|_g = 1 \}$ for $h$ sufficiently small; the following subsection on semiclassical wavefront $WF_{h}$ provides more details.  We are only interested here in $A_h \in \Psi^{m,-\infty}_{cl}$ since all symbols will have compact support in $\xi$ and therefore only the asymptotic behavior in $h$ is relevant.

\subsection{Semiclassical wavefront sets $WF_{h}$}
Let $( \phi_{h})_{h \in (0,h_0]}$ be a family of tempered $L^2$ functions on $M$ in the sense that $\| \phi_{h} \|_{L^2(M)} = {\mathcal O}(h^{-N})$ for some $N>0$ as $h \rightarrow 0^+.$  In the semiclassical case, one is interested in determining decay properties of the family $(\phi_{h})_{h \in (0,h_0]}$, not just regularity properties as functions on $M$. This naturally leads one to  define the notation of a $semiclassical$ $wave$ $front$ $set$,  $WF_{h} (\phi_{h}),$ associated with the family of functions $(\phi_{h})_{h \in (0,h_0]}$; the reader is referred to \cite[Chapter 8, Section 4]{Zw} for further details.  As in the homogeneous case, it is more natural to define the complement of the semiclassical wave front of the family $(\phi_{h})_{h \in (0,h_0]}$ with
\begin{eqnarray} \label{wf}
\nonumber WF_{h} (\phi_{h})^{c} &=& \{ (x,\xi) \in T^*M; \exists a \in S^{0,0}_{cl} \,\, \text{with} \,\,  a(x,\xi) \neq 0 \,\, \\  && \text{and}  \,\, \| Op_{h}(a) \phi_{h} \|_{L^2} = {\mathcal O}(h^{\infty}) \|\phi_{h}\|_{L^2} \}. 
\end{eqnarray}

Finally, we introduce some notation here that is useful when measuring regularity in $h$: we say that a semiclassical family $f_h \in {\mathcal O}_{S}(h^{K})$ provided $\|\partial^{\alpha} f_h \|_{L^{\infty}(\partial \Omega)} = {\mathcal O}_{\alpha} (h^{K}).$

\subsection{Eigenfunction localization} \label{localization}
In order to prove the semiclassical variant of (\ref{conormal}) (see Lemma \ref{basiclemma}), we will use the $h$-microlocal concentration properties of both the eigenfunctions $\phi_h \in C^{\infty}(\Omega)$ and their normalized Dirichlet boundary traces $h \partial_\nu \phi_h \in C^{\infty}(\partial \Omega).$
  We begin $h$-microlocal concentration for the eigenfunctions, $\phi_h$ which amounts to  $h$-microlocally cutting-off homogeneous FIOs near the energy level set $p^{-1}(0) = \{ (x,\xi) \in T^*\Omega; |\xi|_g^2 - 1 =0 \}.$  Such procedures are standard (see \cite[Chapter 8]{Zw}) but, for the convenience of the reader, we briefly review the specific case at hand. By assumption $P(h) \phi_{h} = o(1)$ and since $P(h) = -h^2 \Delta_g - 1$ is clearly $h$-elliptic off $p^{-1}(0).$ Let $\chi \in C^{\infty}_0(\R;[0,1])$ be a standard cutoff function equal to $1$ near the origin  and $\chi_{\Omega} \in C^{\infty}_{0}(\R^n;[0,1])$ with $\chi_{\Omega}(x) = 1$ for $x \in \Omega.$  We define the associated cutoff  
\begin{equation} \label{cutoff0}
\chi_0(x,\xi): = \chi ( p(x,\xi)) \cdot \chi_{\Omega}(x) \end{equation}  and let $\chi_0 (h) := Op_{h}(\chi_0) \in \Psi_{cl}^{0,-\infty}(\R^n)$  be the corresponding $h$-pseudodifferential cutoffs.  It now follows by a standard semiclassical parametrix construction, as in \cite[Theorem 6.4]{Zw}, that
$ \| \phi_{h} - Op_{h}(\chi_{0}) \phi_{h} \|_{L^2(\Omega)} = {\mathcal O}(h^{\infty})$
and so,
\begin{equation} \label{loc}
 WF_{h}( \phi_{h}) \subset p^{-1}(0). \end{equation}

Using a similar argument (see Lemma 48 in \cite{TZ2}), one can prove the following basic $h$-microlocalization result for eigenfunction restrictions. We state the result for general hypersurfaces in arbitrary dimension. We are of course interested here in the particular case where $H= \partial \Omega.$

\begin{lem} \label{resloc}
Let $\Omega$ be a compact manifold (with or without boundary),  $H \subset \Omega$ be a hypersurface and $\gamma_H:C^{0}(\Omega) \rightarrow C^0(H)$ be the restriction operator. Then 
\begin{equation}
WF_{h}(\gamma_H \phi_{h}) \subset B^* H.
\end{equation}
\end{lem} 

We note that  $u_{h}^b(q) = N_{\partial \Omega}^{\nu} ( \gamma_H \phi_{h})$ where $N_{\partial \Omega}^{\nu} :C^{\infty}(\partial \Omega) \rightarrow C^{\infty}(\partial \Omega)$ is the Dirichlet-to-Neumann opeartor. It is well-known that $N_{\partial \Omega}^{\nu}  \in Op(S^{1}_{cl}(T^*\partial \Omega)).$
Thus, by fiber rescaling $\xi \mapsto \xi/h$ it follows  that   $N_{\partial \Omega}^{\nu}(q, h D_q) \in \Psi^{1,1}_{cl}(\partial \Omega)$.  Using semiclassical wave-front calculus \cite{Zw}, it  then follows that
$$WF_{h}(u^b_{h}) = WF_{h}( \, N_{\partial \Omega}^{\nu} ( \gamma_H  \phi_{h} ) \, ) \subset WF_{h}( \gamma_H \phi_{h}).$$
Consequently, from Lemma \ref{resloc} one gets that
\begin{equation} \label{localize}
WF_{h}(u^b_{h}) \subset B^* H.
\end{equation}

From now on we let $\chi_1 \in C^{\infty}_0 (T^*\partial \Omega)$ with
\begin{align} \label{cutoff2}
\chi_1(y,\eta) &= 1 \,\,\, \text{when} \,\, |\eta|_g < \frac{3}{2},  \nonumber \\ 
\chi_1(y,\eta) &= 0\,\,\, \text{when} \,\, |\eta|_g >2. \end{align}
Here $| \cdot |_{g}$ denotes the induced metric on $T^*H.$ 

In the special case of Euclidean domains, we note the wave-front containment relation (\ref{localize}) can also be easily verified directly from the potential layer formulas for the eigenfunctions $\phi_{h}$ \cite{TZ2}.

\subsection{Semiclassical Fourier integral operators ($h$-FIOs)} \label{hFIO}

The classical homogeneous theory of FIOs (and $C^{\infty}$ wavefront sets) was developed in \cite{Ho} and continues to be a useful tool in  scattering theory and spectral asymptotics.  However, due to the compactness of $WF_h(u^b_{h})$ in (\ref{localize}), it is useful here to consider operators  $F_h:C^{\infty}(\mathbb{R}^{n_y}) \to C^{\infty}(\mathbb{R}^{n_x})$ with Schwartz kernels locally of the form
\begin{equation} \label{fiodefn}
F_h(x,y) = (2\pi h)^{- (N/2)  - (n_x/2)} \int_{\R^N} e^{i \psi(x,y; \xi)/h} a(x,y; \xi,h) \, d\xi, \end{equation}
for $a(x,y,\xi, h) \in C^{\infty}_0(U \times V \times \R^N \times (0, h_0])$ with $a(x,y,\xi,h) \sim_{h \rightarrow 0} \sum_{j=0}^{\infty} a_j(x,y; \xi) h^{m+j}$, $a_j(x,y;\xi) \in C^{\infty}_0(U \times V \times \R^N)$, for some $m \in \R$, with $U$ and $V$ being open subsets of $\mathbb{R}^{n_x}$ and $\mathbb{R}^{n_y}$, respectively.  We assume that $\psi$ is a non-degenerate phase function, that is
\begin{equation}
C_{\psi}: = \{ (x,y,\xi) \in U \times V \times \R^N; d_{\xi} \psi(x,y,\xi) = 0 \}
\end{equation}
is a submanifold of $U \times V \times \R^N$ and the set of vectors $\{d_{x,y,\xi}(d_{\xi_1}\psi),...,d_{x,y,\xi}(d_{\xi_n}\psi)\}$ is linearly independent on $C_{\psi}$.

We use this local formulation to define our operator between compact, $n$-dimensional manifolds $M$ and $N$.  As a result of the above oscillatory expression, $F_h$ can be associated to an immersed Lagrangian manifold $\Gamma \subset T^*M \times T^*N$, where
$$ \Gamma = \{ (x,d_x \psi; y, -d_y \psi); d_{\xi} \psi (x,y,\xi) = 0 \},$$ by pushing forward $C_\psi$ through the map $i_{\psi}: C_{\psi} \to \Gamma \subset T^*(M \times N)$ given by $(x, y, \xi) \to (x, d_x \psi, y, -
d_y \psi)$; the manifold $\Gamma$ then contains $WF_{h}(F_h).$  For $h \in (0,h_0]$ small, the operators $F_h: C^{\infty}(N) \rightarrow  C^{\infty}(M)$, with Schwartz kernel locally of the form (\ref{fiodefn}), are called compactly supported {\em $h$-Fourier integral operators}, or semiclassical FIOs, and we write $F_h \in \mathcal{F}^{m,-\infty}_{0} (M \times N; \Gamma).$  It is clear that if $\psi$ is a non-degenerate phase in the sense of H\"ormander \cite{Ho}, then $d=0$.

The symbol of a semiclassical
Fourier integral operator $F_h \in \mathcal{F}^{m,-\infty}_{0}(M \times N, \Gamma)$, where $h \in (0,h_0]$, is a
smooth section of $\Omega_{1/2} \otimes L$, that is, the half-density bundle of $\Gamma$ twisted by the Maslov line bundle.  In terms of the semiclassical Fourier integral representation in (\ref{fiodefn}), it is
the square root $\sqrt{d_{C_{\psi}}}$ of the delta-function on
$C_{\phi}$ defined by $\delta(d_{\xi} \psi)$, transported to
its image in $T^* (M \times N)$ under $\iota_{\psi}$. Concretely, if $(\lambda_1, \dots,
\lambda_n)$ are any local coordinates on $C_{\phi}$, extended as
smooth functions in a neighborhood of $C_{\phi}$, then
$$ d_{C_{\psi}}: = \frac{|d \lambda|}{|\partial(\lambda,
d_{\xi}\psi)/\partial(x,y,\xi)|}, $$ where $d \lambda$ is the
Lebesgue density and $|\partial(\lambda,
 d_{\xi}\psi)/\partial(x,y,\xi)|$ denotes the determinant of the Jacobian matrix. In the semiclassical Fourier integral representation (\ref{fiodefn}), the symbol is
transported from the critical set  $C_{\psi}$
 by the Lagrangian immersion $i_{\psi}$ to the set $\Gamma$. At a point $(x_0, \xi_0, y_0, \eta_0) \in \Gamma$, the principal symbol of $F_h$
is defined to be
$$\sigma(F_h) (x_0, \xi_0, y_0, \eta_0) := h^{m} i^*_{\psi} a_0
\sqrt{d_{C_{\psi}}}. $$  This formulation gives a coordinate-invariant definition of the principal symbol.

The transverse intersection calculus for $h$-FIOs states that if $F_1 \in \mathcal{F}^{m_1,-\infty}_{0} (M \times N_1; \Gamma_1)$ and $F_2 \in \mathcal{F}^{m_2,-\infty}_{0} (N_1 \times Z; \Gamma_2)$ with transversal intersection $(\Gamma_1 \times \Gamma_2 ) \, \cap  \, (T^*M \times \Delta_{T^*(N_1 \times N_1)} \times T^*Z),$  then $F_3 = F_1 \circ F_2 \in \mathcal{F}^{m_1 + m_2,-\infty}_{0} (M \times Z; \Gamma_1 \circ \Gamma_2)$. More generally, when $\Gamma_1 \circ \Gamma_2$ is clean with excess $e>0,$ the composite operator $F_3 \in \mathcal{F}^{m_1 + m_2 - \frac{e}{2},-\infty}_{0} (M \times Z; \Gamma_1 \circ \Gamma_2).$

For a more detailed treatment of the compactly-supported $h$-Fourier integral operators in (\ref{fiodefn}), see \cite[Chapter 8]{GuSt} and \cite[Chapter 10]{Zw}.

\section{Billiard dynamics and multilink canonical relations} \label{multilink}
In this section we review some basic notions used in the microlocal analysis related to billiard dynamics. 
Primarily, we follow here the exposition in \cite{HZ} (see also \cite{Ch}, \cite{GM} and  \cite{TZ1,TZ2}).
\subsection{Canonical relations}  Let $G^t$ denote the billiard flow on $\Omega$, that is, the broken geodesic flow on the domain $\Omega$.  It is well-known that \cite{Ch, GM} by microlocalizing away from the set of glancing directions, 
(i.e. the directions tangent to $\partial \Omega$, in $T^*(\mathbb{R} \times \Omega \times \Omega)$), one can construct a microlocal Dirichlet wave parametrix 
\begin{equation} \label{start}
U_{\Omega}(t,x,y) = \sum_{j=-\infty}^{\infty} U^j_{\Omega}(t,x,y),
\end{equation} 
where $U^j_{\Omega} \in I^{-1/4}(\mathbb{R} \times \mathbb{R}^n \times \Omega; \tilde{\Gamma}^j).$  Here,   $\tilde{\Gamma}^j \subset T^*(\mathbb{R} \times \mathbb{R}^n \times \Omega)$ is the \textit{multilink canonical relation} corresponding to the graph of the broken geodesic flow $G^t$ with $j$ reflections. The sum in (\ref{start}) is locally finite and  the full canonical relation of $U_{\Omega}$ is 
\begin{equation}
\tilde{\Gamma} = \bigcup_{ j \in \mathbb{Z}} \tilde{\Gamma}^j.
\end{equation}
We note that this construction takes place in the ambient space $\mathbb{R} \times \mathbb{R}^n \times \Omega$ and that $\tilde{\Gamma}^j$ is denoted by $\Gamma^j_{\pm}$ in \cite{HZ}. The $\tilde{}$ is added here to emphasize that it is a homogeneous object in contrast to the compactly-supported semiclassical counterpart. In the latter case, we drop the  $ \, \tilde{}$.  Moreover, since we consider the half wave operator $U_{\Omega}(t) = e^{-i t \sqrt{\Delta_{\Omega}}}$ here  (not $\cos t \sqrt{\Delta_{\Omega}}$ or $\frac{ \sin t \sqrt{\Delta_{\Omega}} }{ \sqrt{\Delta_{\Omega}} }$ )  the bicharacteristic  flow has only one direction in time.  Consequently, we drop the $_{\pm}$-notation. Let us  now describe $\tilde{\Gamma}^j$ in more detail.

Let $g^{t}$ denote the geodesic flow in $\mathbb{R}^n$.  We define the \textit{first impact time} of $(x, \xi)$ with $\partial \Omega$ to be $t^1(x,\xi) = \inf \{ t>0: \pi(g^{ t}(x, \xi)) \in \partial \Omega \}$.    Maintaining the notation from \cite{Ch}, we define the \textit{first billiard impact} as $\lambda^1(x, \xi) = g^{t^1(x,\xi)}(x, \xi) \in T^*_{\partial \Omega} \mathbb{R}^n$.  Let $\widehat{(x,\xi)} \in T^*_{\partial \Omega} \mathbb{R}^n$ denote the reflection of $(x, \xi) \in T^*_{\partial \Omega} \mathbb{R}^n$ at the boundary $\partial \Omega$ (the sign of the normal component of $\xi$ is changed under the reflection).   
One can inductively define $t^j$ and $\lambda^j(x, \xi)$  for any $j>1$ by setting  
$t^j(x,\xi) := \inf\{t>0: \pi(g^{t}(\widehat{\lambda^{j-1}(x,\xi)})) \in \partial \Omega \}$ and $\lambda^j(x, \xi) = g^{t^{j}(x,\xi)}(\widehat{\lambda^{j-1}(x, \xi)}) \in T^*_{\partial \Omega} \mathbb{R}^n$.  Furthermore, we are now able to set the \textit{j-th impact time} as $T^{j} (x, \xi)= \sum_{k=1}^j t^k(x, \xi)$ for $j > 0$; these are the exact times of the $j$-th billiard impacts.
Using this notation, we can write the multilink Lagrangian $\tilde{\Gamma}^j$ as 
\begin{equation} \label{multilinklag} \tilde{\Gamma}^j = 
\begin{cases} 
\{(t,\tau, g^{t}(x, \xi), x, \xi) \ : \ \tau =  |\xi|,  \, 0 \leq t < T^1(x,\xi) \},  & \mbox{if } j= 0 \\
\{ (t,\tau, g^{(t - T^j(x, \xi))} \widehat{\lambda^j(x, \xi)}, x, \xi) \ : \ \tau =  |\xi|, \,    T^j(x,\xi) \leq t < t^1( \widehat{\lambda^j(x, \xi)} ) \}, & \mbox{if } j\in \mathbb{N}. 
\end{cases}
\end{equation}
The multilink Lagrangian $\tilde{\Gamma}= \cup_{j \in \Z} \tilde{\Gamma}^j$ given by (\ref{multilinklag})  may be viewed as the graph of the billiard flow 
$G^t$. 
\subsection{The billiard ball map}
\label{bbm}
To describe the billiard ball map $\beta$, we introduce Fermi normal coordinates $(q,x_n)$ along $\partial \Omega$ (see \cite{HZ}).  Let $\xi=(\xi', \xi_n) \in T^*_{(q,x_n)} \mathbb{R}^n$ be the dual fiber coordinates.  Hence, we can write points of $T^*_{\partial \Omega} \mathbb{R}^n$ in the form $(q,0, \xi', \xi_n)$ and set $\tau := \sqrt{|\xi'|^2 + \xi_n^2}$. 

We can now define the \textit{billiard map} $\beta$, and its iterates $\beta^j$, by relating it to the multilink canonical relations $\tilde{\Gamma}^j$ and the broken geodesic flow via the  formula
\begin{equation}
G^{T^j(q,0;\xi)}(q,0,\xi', \xi_n) = (\tau \beta^j(q, \frac{\xi'}{\tau}), \tilde{\xi_n}(q,\xi',\xi_n))
\end{equation}
for $\tilde{\xi_n}(q,\xi',\xi_n)=\tau\sqrt{1-|\beta^j(q, \frac{\xi'}{\tau})|^2}$.

Let us recall that  two different, but closely related fixed point sets are used throughout the paper.  The first is the set of periodic trajectories of length $T$, denoted in Definition \ref{microclean} by $\bar{\Lambda}_T$. It is the fixed point set  of the broken geodesic flow $G^t$ on $S^* \Omega$.  The second is the set of periodic points of the billiard map on $B^* \partial \Omega$ of length $T$, which we denote by $\Lambda_{T}.$  The relation between these two sets was explained in Subsection \ref{sec:main}. For more details on this correspondence, see \cite{GM2}.

\section{Semiclassical  analysis of the boundary wave trace} \label{guts}
\subsection{Basic definitions}
We follow the argument in \cite[Section 2.6]{HZ} with the appropriate semiclassical modifications.  These are necessary because of the fact that the restriction operator $\gamma_{\partial \Omega}: C^{\infty}(\Omega) \to C^{\infty}(\partial \Omega)$ is not a well-defined homogeneous FIO microlocally near  $ 0_{T^*\partial \Omega} \times N^*\partial \Omega \subset T^*\partial \Omega \times T^*\Omega.$   
 We put in the eigenfunction cutoffs in (\ref{cutoff0}) and (\ref{localize}) and apply the compactly supported $h$-FIO calculus in Section \ref{hFIO} to circumvent this 0-section issue.  As we have already pointed out, unlike the case in \cite{HZ}, the conormal directions are very important in our analysis. In fact, to prove Theorem \ref{mainthm} in the case of  domains of constant width discussed in Section \ref{examples}, we will have to  additionally microlocalize the $E^b_{\kappa}(t)$ to arbitrarily small neighbourhods of the zero section $0_{B^*\partial \Omega} \in B^*\partial \Omega$ corresponding to near conormal directions to $\partial \Omega.$  
 
 We briefly review here the basic analogues to the homogeneous theory, emphasizing the  modifications that arise by passing to semiclassical calculus.  The symbol computations are nevertheless essentially the same as in the homogeneous setting since the cutoff $h$-FIO's all have homogeneous phases and the associated Lagrangians  are now bounded subsets of the conic Lagrangians in the homogeneous theory. In particular, checking cleanliness of intersections involves the same computations as in the homogeneous theory. For further details, we refer to \cite[Section 2.3 \& 2.4]{HZ} and \cite[Section 7]{TZ2}.

We now give more details on the wave parametrices. The Dirichlet half wave operator $U_{\Omega}(\cdot): C^{\infty}(\Omega) \rightarrow C^{\infty}(\Omega)$ has Schwartz kernel
$$U_{\Omega}(t,x,y) = \sum_{j=1}^{\infty} e^{-i t \lambda_j} \phi_{\lambda_j}(x) \phi_{\lambda_j}(y)$$
and, as is well known, $U_{\Omega}(\cdot)$ has a distributional trace in $t \in \R$ in the sense that for $\hat{\rho} \in {\mathcal S}(\R),$
$$\int_{\R} \hat{\rho_T}(t) U_{\Omega}(t,x,y) dt \in {\mathcal D}'(\Omega \times \Omega).$$
 Let $\gamma_{\partial \Omega}:\Omega \to \partial \Omega$ be the restriction map. We extend the exterior normal derivative $\partial_{\nu}$ to a $C^{\infty}$ vector field $\partial_{\nu}(x)$ in a neighbourhood $U$ of $\partial \Omega$ in $\R^n.$ As in \cite{HZ}, we let $n^D:= \gamma_{\partial \Omega} \partial_{\nu}: C^{\infty}(\Omega) \to C^{\infty}(\partial \Omega)$ be the Dirichlet boundary trace operator and let  $\tilde{\kappa}(q,D_q,D_t) \in Op(S^0_{cl}(T^*\partial \Omega \times T^* \R ) )$ be a homogeneous order zero pseudodifferential cutoff  defined as follows. For fixed $\epsilon >0,$ we introduce a cutoff function $\kappa$ with the following properties:
\begin{align} \label{supportassumption}
&(i) \, \kappa \in C^{\infty}_{0}(B^*\partial \Omega;[0,1]),  \nonumber \\
&(ii) \, \kappa(y,\eta) = 1  \, \text{ for} \,  (y,\eta) \in \Lambda_T \cup \Lambda_T^{-1},  \nonumber \\
& (iii) \, \kappa(y,\eta) =0 \, \text{for all periodic points of the billiard flow} \, (y,\eta) \notin \Lambda_T \cup \Lambda_T^{-1}.
 \end{align}  
We note that such a cutoff exists in view of the admissibility assumption in Definition \ref{microclean}.  Here $\Lambda_T^{-1}$ denotes the fixed point set of the billiard ball map corresponding to the same trajectories as in $\tilde{\Lambda}_T$, but traversed in the opposite direction. 
  
  For $(q,\eta,t,\tau) \in T^*\partial \Omega \times T^*\R-0,$ we let $\tilde{\kappa}$ be a positive homogeneous order zero cutoff with
 \begin{equation} \label{homsc}
 \tilde{\kappa}(q,\eta,\tau) = \kappa (q, \frac{\eta}{\tau}). \end{equation}

  Then,  in view of the key formula (\ref{key}), one forms  the composite operator $E^b_{F,\tilde{\kappa}}(\cdot): C^{\infty}(\partial \Omega) \rightarrow C^{\infty}(\partial\Omega \times \R)$ 
defined by
\begin{equation} \label{composite}
E^b_{F,\tilde{\kappa}}(t)= \tilde{\kappa}(q,D_q,D_t)  F(q) \gamma_{\partial \Omega} \partial_{\nu} U_{\Omega}(t) \partial_{\nu}^* \gamma_{\partial \Omega}^* \tilde{\kappa}^*(q,D_q,D_t), \end{equation}
where
\begin{equation} \label{key2} \begin{array}{ll}
  \int_{\R} \hat{\rho}(t) e^{it \lambda} \, Tr_{\partial \Omega} E_{F,\tilde{\kappa}}^{b}(t) \, dt = \sum_{j} \rho_T(\lambda_j-\lambda) \| \sqrt{F(q)}\tilde{\kappa}(q,D_q,\lambda_j) u_j^b \|_{L^2(\partial \Omega)}^2 + {\mathcal O}(\lambda^{-\infty}).  \end{array} \end{equation}

As we remarked in the introduction, the assumption that $\Omega$ is star-shaped relative to $0 \in \Omega$ implies that $F(q) = \langle q, \nu(q) \rangle  \geq 0.$
The operator $E^b_{F,\tilde{\kappa}}(t)$ is only a homogeneous FIO when appropriately microlocalized away from the tangential and normal directions, and so the latter microlocal condition  is a necessary requirement in order to apply the composition theorems for homogeneous FIOs to compute $\lambda \to \infty$ asymptotics in (\ref{key2}). We assume that the $\Lambda_T$ contain co-vectors transversal to the boundary (i.e. non-tangential). As in \cite{HZ, TZ1}, we are able to microlocalize away from the glancing set.  However, to make $E^b_{\tilde{\kappa}}(t)$ an actual homogeneous FIO in the sense of H\"ormander, one also has to microlocalize away from the normal codirections to the boundary. In certain  examples (such as the domain of constant width), the relevant fixed point sets $\Lambda_T$ consist {\em entirely} of bouncing-ball conormal vectors to the boundary; therefore, we are not free to carry out the latter microlocalizations to $E^b_{\tilde{\kappa}}(t).$ In order to deal with this point, we instead work in the semiclassical setting of compactly supported $h$-FIOs described in Section \ref{hFIO}.

Indeed, from now on, we let $h = \lambda^{-1}$ and $h_j = \lambda_j^{-1}$, for  $j \in \mathbb{N}$ and using the $WF_h$-localization of the eigenfunctions and their boundary traces in (\ref{loc}) and  (\ref{localize}),  we instead replace $E^b_{F,\tilde{\kappa}}(t)$ by an appropriate semiclassical family of compactly supported $h$-FIOs $E^b_{F,\kappa}(t,h):C^{\infty}(\partial \Omega) \to C^{\infty}(\partial \Omega) $ in the sense of Section \ref{hFIO}, defined in (\ref{cutoff operator}), with the property that
\begin{align} \label{approx}
\sum_{j} \rho_T(h_j^{-1}- h^{-1}) \langle F(q) \tilde{\kappa}(q,D_q,h_j^{-1}) u_j^b, \tilde{\kappa}(q,D_q,h_j^{-1}) u_j^b \rangle_{L^2(\partial \Omega)}^2  \nonumber \\
\sim_{h \to 0^+} \int_{\R} \hat{\rho_T}(t) e^{it/h} \, Tr_{\partial \Omega}E^b_{F,\kappa}(t,h) \, dt. \end{align}
This is precisely the content of Lemma \ref{scapproximation}. The point is that one can then use composition theorems and symbol calculus \cite{GuSt} for such compactly-supported $h$-FIOs to compute the asymptotics of the RHS in (\ref{approx}), bypassing the issue of conormal sets to the boundary that arises in the homogeneous setting.

\subsection{Dirichlet boundary trace} The operator $n^{D}: C^{\infty}(\Omega) \to C^{\infty}(\partial \Omega)$ given by $n^D: u \mapsto \gamma_{\partial \Omega} \partial_{\nu} u$  is not a homogeneous FIO microlocally near $0_{B^*\partial \Omega} \times N^*\partial \Omega$, where $0_{B^*\partial \Omega} = \{(q,0) \in B^*\partial \Omega \}$.  To compensate for this, we consider the semiclassically cutoff operators $n^D(h): C^{\infty}(\Omega) \to C^{\infty}(\partial \Omega)$ with
\begin{equation} \label{cutoff1}
n^{D}(h):= \chi_1(q,h D_q) n^{D} \chi_0(x,h D_x),\end{equation}
where $\chi_{0} \in C^{\infty}_{0}(T^*\R^n)$  (resp. $\chi_1 \in C^{\infty}_{0}(T^*\partial \Omega) )$) are the cutoff functions in (\ref{cutoff0}) (resp. (\ref{cutoff2})). In terms of Fermi coordinates
$$ n^{D}(h)(q,y) = (2\pi h)^{-n} \int_{\R^n} e^{i \langle q-y', \eta' \rangle/h} e^{-i y_n \eta_n/h}  \, i \eta_n \, \chi_1(q,\eta') \chi_0(y,\eta) \, (1 + {\mathcal O}_{{\mathcal S}}(h) )\, d\eta' d\eta_n.$$
 The phase function $\phi(q,y,\eta) = \langle q-y',\eta' \rangle - y_n \eta_n$ is non-degenerate with critical manifold $C_{\phi}= \{(q,y,\eta); q=y', y_n=0\}$ and the  associated Lagrangian is $\Gamma_{\partial \Omega} := \{ (q,\eta';q,\xi) \in \text{supp} \chi_1 \times \text{supp} \chi_0; \,  \xi|_{T_q \partial \Omega} = \eta' \}$, which is just a bounded subset of the corresponding conic Lagrangian in the homogeneous case.   
 
From Section \ref{hFIO}, it is easily checked that
\begin{align*}   
n^D(h) \in \mathcal{F}_0^{-1/2, -\infty}(\partial \Omega \times \Omega; \Gamma_{\partial \Omega}) \end{align*}
Moreover, as in the homogeneous case, in terms of the parametrization $ \iota_{\phi}: T^{*,+}_{\partial \Omega} \Omega \to T^*\partial \Omega \times T^*\Omega$ given by $\iota(q,\eta',\eta_n) = (q,\eta';q,\eta',\eta_n),$  one computes the symbol to be
\begin{align*} 
\sigma (n^D(h))(q,\eta) = i \,  \chi_1(q,\eta') \chi_0(q, \eta) \, \eta_n \, |   dq d\eta' d\eta_n |^{\frac{1}{2}}.
\end{align*}
Consequently, for any $\kappa(h) \in Op_{h}(S^{0,-\infty}),$ $\kappa(q,h D_q) \circ n^D(h) \in \mathcal{F}_0^{-1/2, -\infty}(\partial \Omega \times \Omega; \Gamma_{\partial \Omega}) $ with
\begin{equation}\label{symbol1} 
\sigma ( \kappa(h) n^D(h) )(q,\eta) = i \,  \chi_1(q,\eta') \kappa(q,\eta') \chi_0(q,\eta) \, \eta_n \, |   dq d\eta' d\eta_n |^{\frac{1}{2}}.
\end{equation}

As for the adjoint,
\begin{align} \label{symbol1.1}
n^D(h)^* \in  \mathcal{F}_0^{0, -\infty}(\Omega \times \partial \Omega; (\Gamma_{\partial \Omega})^* ) \end{align}
with symbol
\begin{equation} 
\sigma (n^D(h)^*)(q,\eta) =- i \,  \chi_1(q,\eta') \chi_0(q,\eta) \, \eta_n \, |   dq d\eta' d\eta_n |^{\frac{1}{2}}. \end{equation}
We note that by (\ref{loc}), (\ref{resloc}), and a rescaling of the frequency variables $\xi \to h h_j^{-1} \xi,$  one gets that for $h_j$ with $h h_j^{-1} \sim 1$ as $h \to 0^+,$ 
\begin{align} \label{trace1}
n^{D} (1-\chi_{0}(h)) \phi_{h_j} = n^{D} [ 1- \chi_{0} ( (h h_j^{-1}) h_j ) ]\phi_{h_j} = {\mathcal O}_{{\mathcal S}}(h_j^{\infty} )= {\mathcal O}_{{\mathcal S}}(h^{\infty}). \end{align}
Indeed, by Taylor expansion around $h h_j^{-1} =1,$
$$\chi_0(x, h h_j^{-1} \xi) = \chi_0(x, \xi) + R(x,\xi;h h_{j}^{-1})$$
with
$$ R(x,\xi;h h_j^{-1}):=  (h h_j^{-1}-1) \langle \xi, \nabla_{\xi} \chi_0 (x, \zeta(x,\xi,hh_{j}^{-1}) \xi ),$$
with $\zeta \in C^{\infty}( \text{supp} \chi_0 \times (1/2,3/2)).$ 
When $h h_{j}^{-1} \sim 1$ it follows that $R(x,\xi; h h_j^{-1}) =0$ when $|\xi| =1.$  So, by a standard parametrix construction as in Section 3.3  one shows that when $h h_j^{-1} \sim 1,$
$$ Op_{h} ( R(x,\xi;h h_j^{-1}) )  \phi_h = {\mathcal O}_{\mathcal S}(h^{\infty}).$$
Similar reasoning shows that
\begin{align} \label{trace2}
(1-\chi_1(h)) n^{D}  \phi_{h_j} = [1-\chi_1( (h h_j^{-1}) h_j ) ] n^{D} \phi_{h_j} = {\mathcal O}_{{\mathcal S}}(h_j^{\infty}) = {\mathcal O}_{{\mathcal S}}(h^{\infty}). \end{align}
It follows that when $h h_j^{-1} \sim 1$ as $h \to 0^+,$
\begin{equation} \label{chop1}
n^{D} \phi_{h_j} = n^{D}(h) \phi_{h_j} + {\mathcal O}(h^{\infty}).\end{equation}
We note that unlike the homogeneous case, in the compact semiclassical setting, the zero section $0_{B^*\partial \Omega}$ has no special significance.

\subsection{Cut-off Dirichlet wave operators}
As in the homogeneous case, we  first make a cutoff away from the glancing set $\mathcal{G} = \{(q,x_n;\xi',\xi_n) \in T^*_{\partial \Omega}(\Omega); x_n=0, \xi_n =0 \}.  $ Let $\chi_{\gcal} \in C^{\infty}_{0}(T^*\Omega)$ with $\chi_{\gcal} (q,x_n;\xi',\xi_n) = (1-\chi_{\delta})(|\xi_n|) \chi(x_n)$ where $\chi \in C^{\infty}_0(\R)$ is a standard cutoff as in (\ref{cutoff0}) and $\chi_{\delta} \in C^{\infty}_0(\R)$ has support in the set $\{ \xi_n; |\xi_n| \leq 2 \delta \}$ with $\chi_{\delta}(|\xi_n|) = 1 $ when $ |\xi_n| < \frac{3\delta}{2}.$  Then, in analogy with \cite[Section 11]{TZ2}, we define the semiclassically cutoff wave operators 
\begin{equation} \label{wave0}
U_{\Omega}(t,h) := \tilde{\chi}_{\mathbb{R}}(t, h D_t) [ \chi_{\gcal}((x,h D_x) \chi_0(x, h D_x) U_{\Omega}(t) \chi_0(x, h D_x) \chi_{\gcal}(x,h D_x)],
\end{equation}
where $\tilde{\chi}_{\mathbb{R}}(t,t', h) := (2 \pi h)^{-1} \int_{\mathbb{R}} e^{i(t-t')\tau/ h} \tilde{\chi}(\tau) \hat{\rho_T}(t) \, \ d \tau.$  Here $\tilde{\chi} \in C^{\infty}_0(\R)$ with $\tilde{\chi} = 1$ in a neighbourhood of $\tau =1.$   The usual Chazarain parametrix construction \cite{Ch} gives
\begin{equation} \label{wave1}
U_{\Omega}(t;h) = \sum_{j=1}^{\infty} U_{\Omega}^{j}(t;h) +{\mathcal O}_{L^2 \to L^2}(h^{\infty}), \end{equation}
where $WF'_{h} \, U_{\Omega}^j(t;h) \subseteq \tilde{\Gamma}^{j}_{0},$  and $\tilde{\Gamma}_{0}^{j}:= \{ (x,\xi,y,\eta) \in \tilde{\Gamma}^{j}; (x,\xi) \in \text{supp} \chi_0, (y,\eta) \in \text{supp} \chi_0 \}$  is the canonical relation associated to the graph of the broken geodesic flow with $j$ boundary reflections cutoff near $S^*\Omega \times S^*\Omega$; we refer to Section \ref{multilink} and \cite[Section 2.3]{HZ} for the formal definition and background on canonical relations $\tilde{\Gamma}^{j}$.
In fact, the sum on the RHS of (\ref{wave1}) is  finite with the number of terms less than $N(t,\delta) <\infty$ depending on $(t,\delta) \in \R \times \R^+.$ We note that residual smoothing kernels $K(\cdot) \in C^{\infty}(\R \times \Omega \times \Omega)$ that arise in the usual homogeneous Chazarain consruction yield the ${\mathcal O}_{L^2 \to L^2}(h^{\infty})$ remainder on the RHS of (\ref{wave1}). Indeed, by integration-by-parts, it is easily verified that for any smoothing operator  with kernel $K(\cdot),$
$$ \| \tilde{\chi}_{\mathbb{R}}(t, h D_t) \big[ \chi_{\gcal}((x,h D_x) \chi_0(x, h D_x) K(\cdot)\chi_0(x, h D_x) \chi_{\gcal}(x,h D_x) \big] \|_{L^2 \to L^2}  = {\mathcal O}(h^{\infty}).$$
Using (\ref{wave0}) and (\ref{wave1}), it follows that
\begin{equation}
U_{\Omega}(\cdot,h) \in \mathcal{F}_0^{1/2, -\infty}( (\R \times \Omega)\times \Omega; \tilde{\Gamma}_{0}),
\end{equation}  
for $\tilde{\Gamma}_{0} = \bigcup_{j \in \mathbb{Z}} \tilde{\Gamma}^{j}_{0}.$

\subsection{The semiclassical boundary trace operator $E_{F,\kappa}^{b}(t;h)$}
  
Given $\kappa  \in C^{\infty}_0(B^*\partial \Omega),$ we define the semiclassically cutoff operators
\begin{equation} \label{cutoff operator}
E^{b}_{F,\kappa}(t;h) := F(q) \kappa(q,h D_q) n^D(h) U_{\Omega}(t;h) n^{D}(h)^* \kappa(q,h D_q)^*.\end{equation}
For $(q,0,\eta,\xi_n) \in T^{*}_{\partial \Omega}(\R^n)$, we set $\tau := \sqrt{ |\eta|^2 + \xi_n^2}$ and consider the parametrization 
$\iota_{+}^{j}: \R^+ \times T^*\partial \Omega \to T^*(\R \times \partial \Omega \times \partial \Omega)$ as
\begin{equation}\label{param}
\iota_{+}^{j}(\tau,q,\eta) = ( T^j(q,\xi(q,\eta,\tau)),  \tau,  \tau \beta^j(q,\frac{\eta}{\tau}), q, \eta),\end{equation}
where $\xi(q, \eta, \tau) = \eta + \xi_n \nu_q$ and $| \eta|^2 + \xi_n^2 = \tau^2$.
The following proposition is central to our result.
\begin{prop} \label{mainprop1}
For any $\kappa \in C^{\infty}_{0}(B^*\partial \Omega),$
 $$ E^b_{F,\kappa}(\cdot;h) \in {\mathcal F}^{0, -\infty}_0((\mathbb{R}  \times \partial \Omega) \times \partial \Omega; \Gamma)$$ 
 where $\Gamma = \bigcup_{j \in \mathbb{Z}} \Gamma^{j}$ with
\begin{eqnarray}
\Gamma^{j}&:=& \{ (t, \tau, q, \zeta, q', \zeta ') \in T^*(\mathbb{R} \times \partial \Omega \times \partial \Omega) : \exists \xi \in T_q^*\mathbb{R}^n\, \xi' \in T_{q'}^*\mathbb{R}^n \text{ with }  \\
\nonumber &&(t, \tau, q, \xi, q', \xi') \in \Gamma^j, \xi|_{T_q\partial \Omega} = \zeta,  \xi'|_{T_{q'}\partial \Omega} = \zeta'; \, (q,\zeta) \in \text{supp} \, \kappa \text{ and } (q',\zeta') \in \text{supp} \, \kappa \}.
\end{eqnarray}
Moreover, in terms of the parametrization in (\ref{param}), the symbol of $E^b_{F,\kappa}(\cdot;h)$ restricted to $\Gamma^{j}$ is given by
\begin{align} \label{symbol2}
\sigma( E^b_{F,\kappa}(\cdot;h) )(q,\eta,\tau) &= C_{j}^D  F(q) \, \kappa(q, \frac{\eta}{\tau})  \kappa( \beta^j(q, \frac{\eta}{\tau}) )   \gamma^{1/2}(q,\eta,\tau) \gamma^{1/2}(\tau \beta^j(q,\frac{\eta}{\tau}),\tau ) \,\tau^2 \,\, |dq d\eta d\tau|^{1/2} 
\end{align}
for $\gamma=\sqrt{1-\frac{|\eta|^2}{\tau^2}}$.  In (\ref{symbol2}) $C_{j}^D \neq 0$ are certain non-zero constant Maslov factors.
\end{prop}

\begin{proof} The proof follows  as in \cite[Proposition 4]{HZ} with the revised $h$-FIO operator composition formulas in Section \ref{hFIO}. We note that since $\chi_1|_{B^*\partial \Omega} = 1,$ the cutoff factors $\chi_1(q, \frac{\eta}{\tau}) = \chi(q, \frac{\eta}{\sqrt{|\eta|^2 + \xi_n^2}}) =1$ and therefore drop out of the expression for the symbol. The same is true for the operator $\chi_{\gcal}(h)$, since $\chi_{\gcal} =1$ on supp $\kappa.$ \end{proof}
\subsection{Conormal expansion near $t=T$}
Given Proposition \ref{mainprop1}, the conormal expansion for $ Tr_{\partial \Omega} E^b_{F,\kappa}(\cdot)$ then follows essentially in the same way as in the homogeneous case except that one needs to semiclassically cutoff the homogeneous FIO $\pi_* \Delta^*: C^{\infty}(\R \times \partial \Omega \times \partial \Omega) \to C^{\infty}(\R).$ Following the argument above, for the diagonal embedding $\Delta_{\partial \Omega}: C^{\infty}(\partial \Omega \times \partial \Omega) \to C^{\infty}(\partial \Omega) $ with $\Delta_{\partial \Omega}f(q) = f(q,q),$ we define
 $$\Delta_{\partial \Omega}(h):= \chi_1(q,h D_q)\, \Delta_{\partial \Omega} \,\chi_1(q,h D_q).$$
 Since the kernel is of the form
 $$\Delta_{\partial \Omega}(h)(q,q') = (2\pi h)^{-2(n-1)} \int_{\R^{2(n-1)}} e^{ i \langle q-q^{(1)},\xi \rangle + i \langle q - q^{(2)}, \eta \rangle /h}  \chi_1(q,\xi) \chi_1(q,\eta) \, d\eta d\xi + {\mathcal O}_{L^2 \to L^2}(h),$$
it is easily checked that  $\Delta_{\partial \Omega}(h) \in {\mathcal F}^{-(n-1)/2,-\infty}_{0}(\partial \Omega  \times (\partial \Omega \times \partial \Omega); \Gamma_{\Delta}^{sc})$ with
  the associated canonical relation $\Gamma_{\Delta}^{sc} = \{ (q,\eta + \xi ; q,\eta,q,\xi) \in T^* \partial \Omega \times T^*(\partial \Omega \times \partial \Omega); \,\, (q,\eta +\xi) \in \text{supp}\, \chi_1, \, (q,\eta) \in \text{supp}\, \chi_1, \text{ and } (q,\xi) \in \text{supp}\, \chi_1 \}.$ 
 
Similarily, the fiber integration operator $\pi_*: C^{\infty}(\partial \Omega \times \R) \to C^{\infty}(\R)$ with $\pi_* f(q) = \int_{\partial \Omega} f(q,t) dq$ gets replaced by
 $$ \pi_*(h):= \chi (t-T) \,\chi_1(q,h D_q) \, \pi_{*} \, \chi_1(q,h D_q)$$
 where $\chi(t-T)=1 $ when $t \in \text{supp} \, \hat{\rho_T}.$
 One verifies that $\pi_*(h) \in {\mathcal F}_0^{(n-1)/2,-\infty}( \R \times (\R \times \partial \Omega); \Pi),$ where $\Pi= \{ (t,\tau,t,\tau,q,0) \in T^*(\R \times (\R \times \partial \Omega), t \in \text{supp} \, \tilde{\chi}, \tau \in \text{supp} \, \tilde{\chi} \}.$  Proposition \ref{mainprop1} and the standard $h$ wave front calculus gives, when applied to the Schwartz kernel $E^b_{F,\kappa}(t,q,q';h) \in C^{\infty}(\R \times \partial \Omega \times \partial \Omega; (0,h_0])$,
 
 \begin{align*}
 WF_{h}' ( \pi_*(h) \Delta^*_{\partial \Omega}(h) E^{b}_{F,\kappa}(t,q,q';h) ) \subset \{ (t,\tau) \in T^*(\R); \tau \in \text{supp} \, \tilde{\chi}, t \in \text{supp} \, \chi(\cdot -T), \\ G^t(q,\eta) = (q,\eta) \, \text{for some} \, (q,\eta) \in \text{supp} \, \kappa \}.\end{align*}

 Using the clean intersection calculus in Section \ref{hFIO}, the composition $(\pi_* \Delta_{\partial \Omega}^*) \circ E^b_{F,\kappa}(t,\cdot,\cdot;h)$ is clean with excess $d$   and so,
\begin{equation} \label{upshot11}
 Tr_{\partial \Omega} E^b_{F,\kappa}(t,h):=\pi_*(h) \Delta_{\partial \Omega}^*(h) E_{F,\kappa}^{b}(\cdot,h) \in {\mathcal F}_{0}^{-\frac{d}{2},-\infty}( \R; T^{*,sc}_T(\R)).\end{equation}
 
The corresponding  principal symbol equals \begin{equation} \label{final symbol}
 \sigma ( Tr_{\partial \Omega} E^b_{F,\kappa}(\cdot ,h))(\tau) =  c\, h^{-d/2} \tilde{\chi}(\tau) \, \tau^{d/2 + 2} \, |d\tau|^{1/2}, \end{equation}
 with $T^{*,sc}_T(\R) = \{ (T,\tau ); \tau \in \text{supp} \, \tilde{\chi} \}.$ 
 
 The following proposition characterizes the $h$-Fourier integral distribution $Tr_{\partial \Omega} E^b_{F,\kappa}(t;h)$ and follows from Proposition \ref{mainprop1} and (\ref{final symbol}).
 
 
  \begin{prop} \label{basiclemma}
 Suppose $\partial \Omega$ is a  $C^{\infty}$ bounded domain, and that the cutoff function $\kappa \in C^{\infty}_0(B^*\partial \Omega)$ satisfies 
\eqref{supportassumption}  with  $\Lambda_{T} $  a clean and connected manifold of dimension $d \leq 2n-2.$ Then, as an $h$-Lagrangian distribution,
 $$ Tr_{\partial \Omega} E^b_{F,\kappa}(t;h) = (2\pi h)^{-1-\frac{d}{2}} \, c_{_{\Lambda_T}} \int_{T^{*}_T(\R)} e^{i(t-T)\tau/h } \, \tau^{d/2 +2}( 1 + {\mathcal O}_{{\mathcal S}}(h) ) \tilde{\chi}(\tau) \, d\tau.$$
Here, 
\begin{equation}
\label{CLT}
c_{\Lambda_T} = c_0^D \int_{\Lambda_T\cup\Lambda_T^{-1}} F(q) \gamma(q,\eta) \, d\mu(q,\eta), 
\end{equation} 
where $c_0^D$ is a non-zero constant, $\gamma(q,\eta) = (1-|\eta|^2_{g(q)})^{\frac{1}{2}},$  $g(q)$ denotes the induced boundary metric at $q \in \partial \Omega$ and $d\mu$ is the measure on $\Lambda_T\cup \Lambda_T^{-1}$ induced by symplectic measure $|dq d\eta|$ on $B^*\partial \Omega.$
\end{prop} 
\begin{rem}
\label{remark:Maslov}
 There is a   Maslov factor (the phase shift arising from the stationary phase calculation) embedded in the constant $c_0^D$.
It is a sum of the Maslov coefficients corresponding to  $\Lambda_T$ and $\Lambda_T^{-1}$, which are  nonzero and equal to each other (cf. \cite[Corollary 9]{HZ}). In particular, they do not cancel each other and hence $c_0^D \neq 0$. 
We also remark  that  the assumption in Proposition \ref{basiclemma} that $\Omega$ is star--shaped is important  since it implies that $F(q) \geq 0$ and so never changes sign. This  ensures that $c_{\Lambda_T} \neq 0.$ 
\end{rem}

\subsection{Semiclassical approximation to the boundary trace}
The following simple, but crucial, lemma shows that the semiclassical trace $ \int_{\R} e^{it/h} \hat{\rho_T}(t) Tr_{\partial \Omega} E^b_{F,\kappa}(t;h) \, dt$  is, for our purposes, a suitable approximation to the actual microlocalized boundary trace $$\int_{\R} e^{it/h} \hat{\rho_T}(t)Tr_{\partial \Omega} E_{F,\tilde{\kappa}}^{b}(t) dt.$$ The point here is to essentially  ``undo" the  mass cutoffs $\chi_0$ and $\chi_1$ in the expression $E^b_{F,\kappa}(t;h).$
\begin{lem} \label{scapproximation} Let $\Omega \subset \R^n$ be a smooth bounded domain and $\tilde{\kappa} \in C^{\infty}(T^*\partial \Omega \times T^*\R - 0)$ be a  positive homogeneous function satisfying (\ref{supportassumption}) with 
$\tilde{\kappa}(q,\eta,\tau) = \kappa(q,\frac{\eta}{\tau})$
for some $\kappa \in C_{0}^{\infty}(B^*\partial \Omega).$  Then, with the compactly supported $h$-FIO operators $E_{F,\kappa}^{b}(t;h): C^{\infty}(\partial \Omega) \to C^{\infty}(\mathbb{R} \times \partial \Omega)$ in (\ref{cutoff operator}),
$$\int_{\R} \hat{\rho_T}(t) e^{it/h} \, Tr_{\partial \Omega}E^b_{F,\tilde{\kappa}}(t) \, dt =   \int_{\R} \hat{\rho_T}(t) e^{it/h} \, Tr_{\partial \Omega}E^b_{F,\kappa}(t,h) \, dt  + {\mathcal O}_{\Omega}(h^{-n+2}). $$
\end{lem}
\begin{proof}
We write
\begin{align} \label{lemma1}
\int_{\R} \hat{\rho_T}(t) e^{it/h} \, Tr_{\partial \Omega}E^b_{F,\tilde{\kappa}}(t) \, dt =  \sum_j \rho_T( h_j^{-1} - h^{-1}) \langle \sqrt{F(q)}\tilde{\kappa}(q,D_q,h_j^{-1}) u_{h_j}^{b},  \sqrt{F(q)}\kappa(q,D_q,h_j^{-1}) u_{h_j}^b \rangle^{2}_{L^2(\partial \Omega)} \nonumber \\
= \sum_{j}  \rho_T( h_j^{-1} - h^{-1}) \langle  \sqrt{F(q)}\kappa(q, h_j D_q) u_{h_j}^{b},  \sqrt{F(q)}\kappa(q, h_jD_q) u_{h_j}^b \rangle^{2}_{L^2(\partial \Omega)}.
 \end{align}
 One rescales fiber variables in $\kappa(q,\eta)$ according to the rule $\eta \mapsto h h^{-1}_j \eta$ and makes the Taylor expansion
  \begin{equation} \label{taylor}
  \kappa(q, h_j h^{-1} \eta) = \kappa(q,\eta) + {\mathcal O}_{\mathcal S} ( h_j h^{-1} - 1). \end{equation}
  Since  $ Op_{h_j} (\kappa(q,\eta)) = Op_h ( \kappa(q, h_j h^{-1} \eta)),$ it follows from (\ref{taylor}) and $L^2$-boundedness that the last line in (\ref{lemma1}) equals
\begin{align}\label{lemma1contd}  
  \sum_{j}  \rho_T( h_j^{-1} - h^{-1})   \langle \sqrt{F(q)} \kappa(q, h D_q) u_{h_j}^{b},  \sqrt{F(q)}\kappa(q,hD_q) u_{h_j}^b \rangle^{2}_{L^2(\partial \Omega)} \nonumber \\ + {\mathcal O}(h) \sum_j    |h_j^{-1} - h^{-1}|  \, | \rho_T( h_j^{-1} - h^{-1})  |  \| u_{h_j}^b \|_{L^2(\partial \Omega)}^2. \end{align}
  By \cite{HT} we have that $  \| u_{h_j}^b \|_{L^2(\partial \Omega)}^2 = {\mathcal O}_\Omega (1)$ and so, by leading-order Weyl asymptotics,
   $${\mathcal O}(h) \sum_j    |h_j^{-1} - h^{-1}|  \, | \rho_T( h_j^{-1} - h^{-1})  |  \| u_{h_j}^b \|_{L^2(\partial \Omega)}^2 = {\mathcal O}_{\Omega}(h^{-n+2}).$$
  Since $\rho_T  \in {\mathcal S}(\R),$ it follows that the sum in (\ref{lemma1}) equals
\begin{align} \label{lemma2}
 \sum_{j; |h_j - h| \leq h^{3/2} }  \rho_T( h_j^{-1} - h^{-1}) & \langle  \sqrt{F(q)} \kappa(q, h D_q) \chi_1(q,h D_q) u_{h_j}^{b},  \sqrt{F(q)}\kappa(q, h D_q) \chi_1(q,h D_q) u_{h_j}^b \rangle^{2}_{L^2(\partial \Omega)} \nonumber \\
 & + {\mathcal O}_{\Omega}(h^{-n+2}),\end{align}
 where $\chi_1$ is the cutoff in (\ref{cutoff2}). For $\{j; |h_j - h| \leq h^{3/2} \},$ we insert the estimate
 $$ u_{h_j}^b =  n^D \chi_0(x,h D_x) \phi_{h_j}  + {\mathcal O}(h^{\infty})$$
 in (\ref{lemma2}). Finally, the truncated sum in (\ref{lemma2}) can, modulo ${\mathcal O}(h^{\infty})$ error, be replaced by the sum over all $j \geq 1$  and  so, it follows that the RHS of (\ref{lemma2}) equals
$$ \int_{\R} \hat{\rho_T}(t) e^{it/h} \, Tr_{\partial \Omega} E^b_{F,\kappa}(t,h)  \, dt  + {\mathcal O}_{\Omega}(h^{-n+2}),$$
where $E^b_{F,\kappa}(\cdot;h) \in {\mathcal F}^{0, -\infty}_0((\mathbb{R} \times \partial \Omega) \times \partial \Omega ; \Gamma)$ are the $h$-FIO's in Proposition \ref{mainprop1}.

\end{proof}



In Theorem \ref{mainthm}, we assume that $n=2$ and $d=1.$ From the expansion in Proposi\-tion~\ref{basiclemma}, Lemma \ref{scapproximation} and (\ref{key2}), one gets that 
\begin{align} \label{pretrace}
\sum_j \rho_T(h^{-1}-h_j^{-1}) \|  \sqrt{F(q)}\tilde{\kappa}(q,D_q,h_j^{-1}) u_{h_j}^b \|_{L^2(\partial \Omega)} &= \int_{\R} \hat{\rho_T}(t) e^{it/h} \, Tr_{\partial \Omega}E^b_{F,\kappa}(t,h) \, dt + {\mathcal O}(1) \nonumber   \\
& \hskip -3truecm \sim_{h \to 0^+} (2\pi h)^{-1-\frac{1}{2}} c_{\Lambda_T} \int_{\R} \int_{\R} e^{it/h} \, e^{i (T-t) \tau/h } \, \tilde{\chi}(\tau) \tau^{5/2} \, \hat{\rho_T}(t) \, d\tau dt \nonumber  \\ &  \hskip -3truecm  \sim_{h \to 0}  c_{\Lambda_T} e^{iT/h} h^{-1/2} .  \end{align}
 The last line in (\ref{pretrace}) follows by stationary phase in $(t,\tau) \in \R^2$ and the constant $c_{\Lambda_T}$ is given in Proposition \ref{basiclemma}.
\section{Proof of the Theorem \ref{mainthm}} \label{mainproof}
\subsection{Dirichlet boundary conditions} 
\label{sec:dirichlet}
Recall that 
\begin{equation} \label{counting}
N(\lambda) = \frac{\text{Area}(\Omega)}{4 \pi} \lambda^2 - \frac{\text{Length}(\partial \Omega)}{4\pi} \lambda + R(\lambda),
\end{equation}
Let, as before,   $\rho_T \in S(\R)$ with $\hat{\rho_T} \in C^{\infty}_{0}(\R),$ $\hat{\rho_T}(T) = 1$  and $0 \notin \text{supp} \,  \hat{\rho_T}$.
Combining (\ref{counting}) with the identity
\begin{equation} \label{gen1}
\sum_j \rho_T(\lambda - \lambda_j) = \int_{\R} \rho_T(\lambda - \mu) \, dN(\mu),
\end{equation}
we can rewrite (\ref{gen1})  as
\begin{align} \label{tracedecomp}
\int_{\mathbb{R}} \rho_T(\lambda - \mu) dN(\mu) &= - \frac{\text{Area}(\Omega)}{4 \pi} \int_{\mathbb{R}} \rho_T'(\lambda - \mu)  \, \mu^{2} \ d\mu + \frac{\text{Length}(\partial \Omega)}{4\pi}  \int_{\mathbb{R}} \rho_T'(\lambda - \mu)  \mu \ d\mu  \nonumber \\
 & + \int_{\mathbb{R}} \rho_T(\lambda - \mu) \, dR(\mu).
\end{align}
For the first two terms on the RHS of (\ref{tracedecomp}), we have 
$$\int_{\mathbb{R}} \rho_T'(\lambda - \mu) \ \mu^2 \ d\mu = \int_{\mathbb{R}} \rho_T'(\lambda - \mu)  \mu \ d\mu = 0.$$ This follows since for $m\ge 1$, $\int_{\mathbb{R}} \rho_T'(\lambda - \mu) \ \mu^{m} \ d\mu = (D_s)^m[ \int_{\mathbb{R}}  e^{is\mu} \rho_T'(\lambda - \mu)  \ d\mu] |_{s=0} = (D_s)^m[ e^{-is\lambda}   \widehat{\rho_T'}(s) ] |_{s=0} = 0,$  since $\widehat{\rho_T'}(s) = 0$ for $s$ near $0.$ Thus, from (\ref{gen1}), it follows that 
\begin{equation} \label{guts1}
\sum_{j} \rho_T(\lambda -\lambda_j) = \int_{\mathbb{R}} \rho_T(\lambda - \mu)dR(\mu) = - \int_{\mathbb{R}} \rho_T'( \lambda - \mu) R(\mu) \ d\mu.
\end{equation}


Under the cleanliness  assumption on $\Lambda_T$ and in view of Proposition \ref{basiclemma} and Lemma \ref{scapproximation},
\begin{align} \label{gut2}
\sum_{j} \rho_T(\lambda_j - \lambda) = &\sum_j \rho_T(\lambda-\lambda_j) \langle  \sqrt{F(q)}\tilde{\kappa}(q,D_q,\lambda_j) u_j^b, \sqrt{F(q)} \tilde{\kappa}(q,D_q,\lambda_j) u_j^b \rangle_{L^2(\partial \Omega)}  + O(\lambda^{-\infty}) \nonumber \\
&= \sum_j \rho_T(\lambda-\lambda_j) \langle  \sqrt{F(q)} \kappa(q,\lambda^{-1}D_q) u_j^b,  \sqrt{F(q)}\kappa(q,\lambda^{-1}D_q) u_j^b \rangle_{L^2(\partial \Omega)}  + {\mathcal O}(1)\nonumber \\
  &= \int_\R e^{i\lambda t} \hat{\rho_T}(t) Tr_{\partial \Omega}E^b_{F,\kappa}(t,\lambda^{-1}) \,dt + {\mathcal O}(1)  \nonumber \\
  &= c_{\Lambda_T} e^{iT \lambda} \sqrt{\lambda} + {\mathcal O}(1).\end{align}
Here, $c_{\Lambda_T}$ is a nonzero constant defined by \eqref{CLT}. 
Finally, from the last line of (\ref{gut2}),
\begin{equation} \label{UPSHOT}
- \int_{\R}  \rho_T'(\lambda -\mu) \, R(\mu) \, d\mu = \sum_{j}  \rho_T(\lambda-\lambda_j) \gg\sqrt{\lambda},
\end{equation}
and an application of Lemma \ref{calclemma} completes the proof of Theorem \ref{mainthm} in the Dirichlet case.  
\subsection{Neumann boundary conditions}
\label{Neumann}
In the Neumann case, using a variant of the Rellich identity found in \cite{CTZ}, we obtain:
\begin{equation} \label{rellich/neumann}
2 = \langle F(q) (I + h^2 \Delta_{\partial \Omega}) u^b_h, u_h^b \rangle_{L^2(\partial \Omega)} + \langle R(h) u_h^b, u_h^b \rangle_{L^2(\partial \Omega)}. \end{equation}
Here $\Delta_{\partial \Omega}: C^{\infty}(\partial \Omega) \to C^{\infty}(\partial \Omega)$ is the induced Laplacian on the boundary $\partial \Omega,$  and  the remainder $R(h)$ is an $h$-differential operator of the form $ h a_1(q) (h \partial_{\nu_q} ) +  ha_2(q)(h \partial_{T_q})$ where $T_q$ is unit tangential vector field along $\partial \Omega$ and $a_j \in C^{\infty}(\partial \Omega); j=1,2.$  Since in the  Neumann case $\partial_{\nu_q} \phi_h |_{\partial \Omega} = 0$ and $WF_h( \phi_h|_{\partial \Omega}) \subset B^*\partial \Omega,$ it follows  by $L^2$-boundedness that for planar domains
$ \| R(h) \|_{L^2(\partial \Omega) \to L^2(\partial \Omega)} = O(h)$ and so, 
$$ \sum_{h_j^{-1} \in [ h^{-1}, h^{-1} + 1]} \langle R(h_j) u_{h_j}^b, u_{h_j}^b \rangle_{L^2(\partial \Omega)} = O(1).$$
This is lower-order than the putative leading term  (which is $\sim h^{-1/2}$) and  can be ignored.

The proof of Theorem \ref{mainthm} then follows as in the Dirichlet case with a  minor difference: The multiplicative factor of $\gamma(q,\eta) = \sqrt{1-|\eta|^2}$ gets replaced by $\gamma^{-1}(q,\eta)$ in the symbol formulas in Proposition \ref{mainprop1}. This is due to the fact that in the Dirichlet case the boundary traces involve normal derivatives and  therefore the symbol of the wave parametrix involves an additional factor of $\gamma \times \gamma $ when compared with the Neumann case.  In the last step in (\ref{gut2}), we use the fact that
$\sigma (1+h^2 \Delta_{\partial \Omega})(q,\eta) = \gamma^2(q,\eta)$ to compute the analogue of the constant $c_{\Lambda_T}$. In the Neumann case, it is  equal to
$$c_0^N \int_{\Lambda_T\cup\Lambda_T^{-1}} F(q) \,  \, \sigma(1+h^2\Delta_{\partial \Omega})(q,\eta)  \, \gamma(q,\eta)^{-1} d\mu$$
$$=  c_0^N \int_{\Lambda_T\cup\Lambda_T^{-1}} F(q) \,  \gamma(q,\eta) d\mu, $$
which is again non-zero under the star-shaped assumption on $\Omega$ and the admissibility assumption on $\Lambda_T.$  This completes the proof of Theorem \ref{mainthm}.

\section{Examples} \label{examples}

\subsection{Elliptical billiards} \label{ellipse}
The billiard flow on an ellipse is a classical example of an integrable dynamical system.  Let $\Gamma_{m,n}$ denote the family of periodic orbits with 
$n$ vertices and the winding number $m \le n/2$. It is well-known that the trajectories belonging to each family $\Gamma_{m,n}$ have the same length 
$l(\Gamma_{m,n})=l_{m,n}$ and are tangent to a certain caustic which is either a confocal ellpise or a hyperbola (see \cite{Tab, GM2}). An example of this phenomenon (which is a particular case of the Poncelet porism) is  illustrated by Figure 1. 
\begin{prop} 
\label{ellprop}
Any ellipse satisfies the assumptions of Theorem \ref{mainthm}.
\end{prop}
\begin{proof}
Clearly, each family $\Gamma_{m,n}$ is separated 
from the glancing set.  Moreover, it was shown in \cite[Proposition 4.3]{GM2} that the fixed point set of the iterated billiard ball map $\beta^k$ is clean for any $k$. Therefore, the second and  the third condition of  Definition \ref{microclean} are satisfied by each family $\Gamma_{m,n}$. 
\begin{figure}
\centering
\includegraphics[height=5cm]{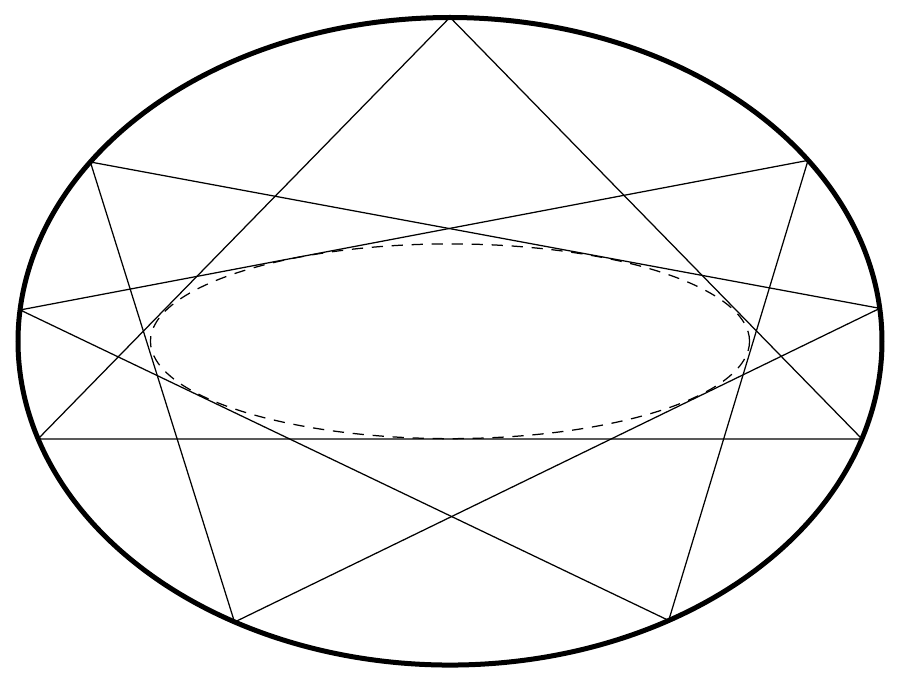}
\caption{Periodic orbits on an ellipse from  the family $\Gamma_{1,3}$.}
\end{figure}
In order to prove that the ellipses satisfy  the assumptions of Theorem \ref{mainthm}, it suffices to show that the first condition  of Definition \ref{microclean} holds for at  least  one pair $m,n$.  This follows from the lemma below. 
\begin{lem}
Let $\Omega$ be an ellipse. For all $n\in \mathbb{N}$ with possibly finitely many exceptions,  $l_{1,n}$ is an isolated point in the length spectrum of $\Omega$,
and $l_{1,n}=l_{m,n'}$ iff $m=1$ and $n=n'$.
\end{lem}
\begin{proof} Let $L=L(\partial \Omega)$ be the perimeter of $\Omega$. 
It follows from the Birkhoff theorem (see \cite[Theorem 2.1]{Sib})  that the sequence $l_{1,n}$ is strictly monotone increasing in $n$ and converging to $L$.
At the same time, by \cite[Proposition 5]{Am}, there exists a number $M$ such that for all $n$ and all $m$ satisfying  $M< m \le \frac{n}{2}$,  one has $l_{m,n}>L+1$.
Moreover, by \cite[Proposition 4]{Am}, for each $m=2,\dots,M,$ there exists $N_m$, such that for all $n>N_m$ and $n\ge 2m$,  
$l_{m,n}>mL- \frac{L}{3}>L+\frac{L}{3}$.
Therefore, for $m\ge 2$, there is possibly a finite number of families of trajectories with lengths smaller than $L$.
Since  $ l_{1,n}<L$ for any $n \ge 2$  and $l_{1,n}$ is strictly monotone increasing in $n$, we deduce that for all but possibly finitely many $n$,  
the number  $l_{1,n}$ is an isolated point in the length spectrum and 
$l_{1,n} =l_{m,n'}$ iff $m=1$ and $n=n'$. 
This completes the proof of the lemma.
\end{proof}
Set now $T=l_{1,n'}$ and consider the fixed point set of the iterated billiard ball map $\beta^{n'}$,  corresponding to the family
$\tilde \Lambda_T=\Gamma_{1,n'}$.  As was shown above, it satisfies the assumptions of Theorem \ref{mainthm}. This completes the proof of the proposition.
\end{proof}
\begin{rem} The eigenvalue counting problem on an ellipse  can be reduced to a lattice counting problem.  As a consequence,  
one has a van der Corput-type remainder estimate  $R(\lambda)=O(\lambda^{2/3})$ (\cite{Kuz},  see also \cite{KF, CdV} for this result on a disk).
To our knowledge, the lower bound \eqref{main:bound} is new even for a disk. We conjecture that it is sharp (on a polynomial scale) for any ellipse, and that the 
optimal upper bound should be the same as in Hardy's conjecture for the Gauss's circle problem: $R(\lambda)=O\left(\lambda^{\frac{1}{2}+\varepsilon}\right)$
for any $\varepsilon>0$.
\end{rem}

\subsection{Domains of constant width} 
Let $\Omega$ be a domain of constant width.  The billiard ball map on $\Omega$ has a one-parameter family of closed trajectories formed by
the bouncing ball orbits at each point $q\in \partial \Omega$ (see \cite{Tab}).  While many domains of constant width  are piecewise smooth (like the Reuleaux polygons and  other domains of constant width  whose boundary is a union of cicrular arcs), there exist a large number of smooth domains of constant width.
\begin{figure}
\centering
\includegraphics[height=5cm]{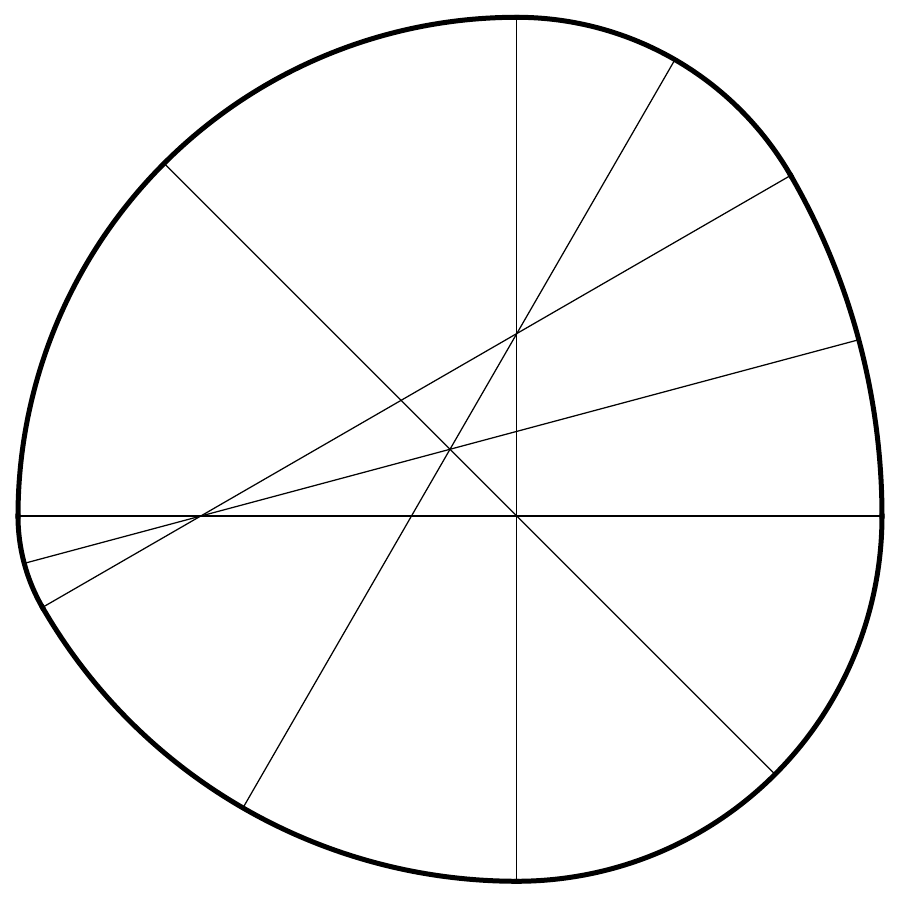}
\caption{Bouncing ball orbits on a domain of constant width.}
\end{figure}
Indeed,  by an arbitrarily small perturbation,  any domain of constant width  could be transformed into a smooth  (and even analytic) domain of the same 
constant width (see \cite{Weg} and references therein).
\begin{prop}
\label{cwidth}
Any smooth domain of constant width satisfies the assumptions of Theorem \ref{mainthm}.
\end{prop}
\begin{proof}
Let $\Omega$ be  a smooth domain of constant width $W$. The set  $\Lambda_{2W}=\{(q,\eta)|\, q \in \partial \Omega, \eta=0\}$ is a fixed point set of the iterated billiard ball map $\beta^2$ corresponding to the bouncing ball orbits. It is a one-dimensional submanifold of $B^*{\partial \Omega}$.
Let us show that $\Lambda_{2W}$ is a clean fixed point set of $\beta^2$ in the sense of Definition \ref{clean}. The differential of $\beta^2$ was computed in \cite[p. 3]{KC}. We have:
\begin{equation} 
\label{ww}
D \beta^2_{(q, 0)} =
\begin{pmatrix}
1 & -\frac{2 W R_q}{W-R_q} \\
0 & 1
\end{pmatrix},
\end{equation}
where $W$ is the width of $\Omega$ and $R_q$ is the curvature radius of $\partial \Omega$  at the point $q$. Note that for a domain of constant width, 
$W=R_q+R_p$, where $p$ is diametrically opposite to $q$.  The fixed point set of 
$D\beta^2$ is given by $\eta=0$. Hence,  $\Lambda_{2W}$ is clean and the assumption (ii) of Definition \ref{microclean} is satisfied.  The assumption (iii) is 
 satisfied as well, since  the bouncing ball orbits are away from the glancing set. The validity of the assumption (i) follows from the lemma below, which 
completes the proof of Proposition \ref{cwidth}.
\end{proof}
\begin{lem}
Let $\Omega$ be a domain of constant width $W$. Then the bouncing ball orbits are the only shortest periodic orbits in $\Omega$. Moreover, there exists $\delta>0$ such that any other periodic orbit has length at least $2W+\delta$.
\end{lem}
\begin{proof}
It follows from \cite[Theorem 1.1]{BB} that any shortest periodic orbit on $\Omega$ is either a bouncing ball orbit or a triangular orbit.  Let us show that any triangular orbit $ABC$ has length greater than $2W$.  Our proof follows closely the ideas of \cite{BB}.  Let $A'$, $B'$, $C'$ be the diametrically opposite points  to, respectively, $A$, $B$, $C$. Consider a disk-polygon $\mathcal{P}$  which is equal to  the intersection of the disks of radius $W$ centered at $A'$, $B'$ and $C'$. It is easy to see that $ABC$ is a periodic billiard orbit for $\mathcal{P}$ as well. Moreover, $\Omega \subset \mathcal{P}$  and hence the width $W'$ of $\mathcal{P}$ is greater or equal 
than $W$.   By \cite[Theorem 1.2]{BB},  the orbit $ABC$ (as an orbit in the disk-polygon $\mathcal{P}$) has length greater than $2W'$, and hence greater than $2W$.

 Recall the notation of subsection \ref{ellipse}: let $\Gamma_{m,n}$  be the collection of orbits $\gamma$ that hit the boundary $n$ times and have 
the winding number $m$.  For any periodic orbit $\gamma$, let  $l(\gamma)$ be its length. 
Now, suppose there exists a sequence of periodic orbits $\gamma_k \in \Gamma_{m_k,n_k}$ such that $\l(\gamma_k) \to 2W$ as $k \to \infty$.  
Since any domain of constant width is strictly convex, one can apply \cite[Proposition 5]{Am}, which immediately implies that the sequence $m_k$ is bounded.
Therefore, for some $m_0$, one may choose a subsequence $\gamma_{k_j} \in \Gamma_{m_0,n_{k_j}}$ such that $\l(\gamma_{n_{k_j}}) \to 2W$ as $j \to \infty$. 
By \cite[Proposition 4]{Am} we get that the sequence $n_{k_j}$ is also bounded, since otherwise 
$\lim_{j\to\infty} l(\gamma_{n_{k_j}})= m_0L(\partial \Omega) > 2W.$  Therefore, for some $n_0$,  there exists a sequence of orbits $\gamma_i \in \Gamma_{m_0,n_0}$ 
such that $\lim_{i\to\infty} l(\gamma_i) = 2W$. The closure of the set $\Gamma_{m_0,n_0}$ is 
a compact set. Hence, the sequence $\gamma_i$ has a limit which  is itself a periodic orbit which hits the boundary at most $n_0$ times, has the 
winding number $m_0$ and length $2W$. As was shown above, it must be  a bouncing ball orbit. This implies that $m_0=1$ and $n_0=2$, since  the only periodic orbits that may converge to a bouncing ball orbit are the bouncing ball orbits.
This completes the proof of the lemma.
\end{proof}
\begin{rem}
In \cite{PT} we aim to extend the inequality \eqref{main:bound} to piecewise smooth domains of constant width with $C^{1,1}$ boundary.  
It is a challenging problem to find an approach  covering  the case of Reuleaux polygons.
Some other questions related to spectral asymptotics on domains of constant width have been recently considered in \cite{Gut, Hor}.
\end{rem}
\subsection{Balls in $\mathbb{R}^n$}
\label{balls}
As was mentioned in the introduction, for most domains of dimension $n \ge 3$,  Theorem \ref{higher} yields  a better lower bound on the error term in Weyl's law than the multi-dimensional analogue of Theorem \ref{mainthm} given by \eqref{dyn:higher}. However,  as we prove  below, this is not the case for Euclidean balls.
\begin{figure}
\centering
\includegraphics[height=5cm]{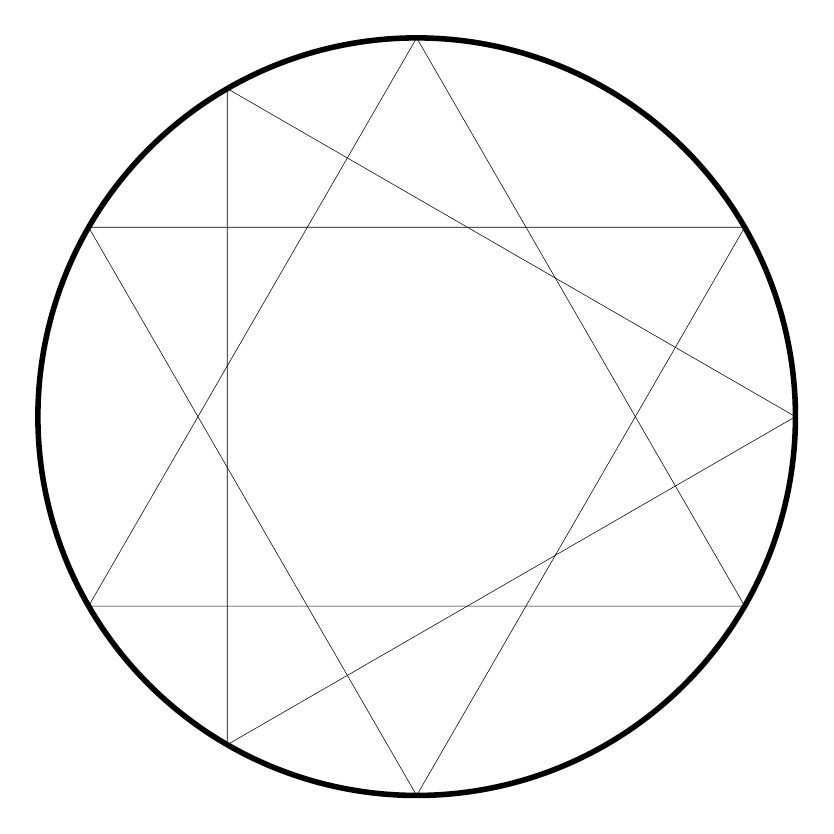}
\caption{Triangular periodic orbits in a $2$-dimensional section of a ball.}
\end{figure}

\noindent {\it Proof of Proposition \ref{ball:bound}.}
In order to check that the proof of Theorem \ref{mainthm} generalizes to the $n$--dimensional ball, one needs to verify  the analogues of the three assumptions stated in Definition \ref{microclean}.   It is well known that each billiard trajectory in a Euclidean ball is contained in a two-dimensional plane passing through the center of the ball. Therefore,
the periodic billiard trajectories are the same as for the circular billiard. For any $k\ge 3$, consider the fixed point set of the iterated billiard ball map $\beta^k$ 
given by the regular $k$-gons inscribed in all possible two-dimensional sections of the ball.  
The dimension of this fixed-point set is equal to  $2n-3$ \cite{Bab, FT}.  Indeed, for each $2$-plane we get a one-dimensional family of such trajectories,  and  the dimension of the corresponding Stiefel manifold $V_2(\mathbb{R}^n)$ is equal to  $2n-3$.  This fixed point set  is a  critical manifold of the corresponding length function, which is nondegenerate in the sense of Bott \cite[Proposition 5.1]{FT}, and hence the fixed point set is clean (cf. \cite[Proposition 4.6]{SZ}).   Therefore, the analogue of condition (ii) in Definition \ref{microclean} is satisfied. Condition (iii) is also trivially satisfied, as the trajectories considered above stay 
away from the glancing set. It remains to check that there exists $k \ge 3$ such that the length of the corresponding periodic trajectories 
is isolated in the length spectrum in the sense of Definition \ref{microclean}.  In subsection \ref{ellipse} we proved this for any ellipse, so in particular it holds for a circle.  In fact, using explicit formulas for the lengths of periodic trajectories on a circle, it could be shown  that one can take simply $k=3$.
Therefore,  formula \eqref{dyn:higher} holds with $d=2n-3$. This completes the proof of the proposition.
\qed


\begin{thebibliography}{HHHH}
\bibitem[Kuz]{Kuz} N.V.  Kuznecov, {\it Asymptotic distribution of the eigenfrequencies of a plane membrane in
the case when the variables can be separated},  Differential Equations 2 (1966), 715--723.
\end{thebibliography}
\end{document}